\documentclass[a4paper,12pt]{amsart}
\usepackage[foot]{amsaddr}
\usepackage[T1]{fontenc}
\usepackage[utf8]{inputenc}
\usepackage{palatino}
\usepackage{amsmath, amssymb,mathtools}
\usepackage{soul}
\usepackage[usenames,dvipsnames]{xcolor}
\usepackage{graphicx}
\usepackage{subfigure}
\usepackage{minitoc}
\usepackage{tikz}
\usepackage{enumerate}
\usepackage[margin=0.96 in]{geometry}
\usepackage{bbm}
\usepackage[numbers,sort&compress]{natbib}
\usepackage[colorlinks=true]{hyperref}
\hypersetup{urlcolor=blue, citecolor=red}
\usepackage{microtype}
\usepackage{crossreftools}
\usepackage{wrapfig}

\DeclarePairedDelimiter{\abs}{\lvert}{\rvert}
\DeclarePairedDelimiter{\norm}{\lVert}{\rVert}
\DeclarePairedDelimiter{\set}{\{}{\}}

\DeclareMathAlphabet{\mathup}{OT1}{\familydefault}{m}{n}
\newcommand{\dx}[1]{\mathop{}\!\mathup{d} #1}

\DeclarePairedDelimiter{\prt}{(}{)}
\DeclarePairedDelimiter{\brk}{[}{]}

\newcommand{\N}{{\mathbb N}}
\newcommand{\R}{{\mathbb R}}
\newcommand{\Rd}{{\mathbb R^d}}
\newcommand{\curlyH}{\mathcal{H}}

\usepackage{amsthm}
\theoremstyle{plain}
\newtheorem{theorem}{Theorem}[section]
\newtheorem{lemma}[theorem]{Lemma}
\newtheorem{proposition}[theorem]{Proposition}
\newtheorem{corollary}[theorem]{Corollary}

\theoremstyle{remark}
\newtheorem{remark}[theorem]{\bf Remark}
\newtheorem{definition}[theorem]{\bf Definition}

\newcommand{\ie}{\emph{i.e.}}

\renewcommand{\i}{^{(i)}}

\newcommand{\epsnun}{_{\epsilon,\nu,N}}
\newcommand{\nun}{_{\nu,N}}
\newcommand{\oun}{_{0,N}}
\newcommand{\nnun}{_{\nu_N,N}}
\newcommand{\oinfty}{_{0,\infty}}

\newcommand{\sign}{\mathrm{sign}}
\newcommand{\ds}{\displaystyle}
\newcommand{\ddt}{\frac{\dx{}}{\dx t}}
\newcommand{\partialt}[1]{\frac{\partial #1}{\partial t}}

\newcommand{\fpartial}[1]{\frac{\partial}{\partial #1}}
\newcommand{\grad}{\nabla}
\renewcommand{\div}{\nabla\cdot}

\newcommand{\Lap}{\Delta}
\newcommand{\indicator}{\mathbbm{1}}
\DeclareMathOperator\supp{supp}
\newcommand{\calK}{\mathcal{K}}

\title[]{From Finite to Continuous Phenotypes in (Visco-)Elastic Tissue Growth Models}

\author{Tomasz D\k{e}biec$^{1}$}
\author{Mainak Mandal$^{2}$}
\author{Markus Schmidtchen$^{2}$}

\address{$^{1}$ Institute of Applied Mathematics and Mechanics, University of Warsaw, Banacha 2, 02-097 Warsaw, Poland (t.debiec@mimuw.edu.pl).}
\address{$^{2}$ Institute of Scientific Computing, Faculty of Mathematics, TU Dresden, Zellescher Weg 12-14, 01069 Dresden 
	(mainak.mandal@tu-dresden.de, markus.schmidtchen@tu-dresden.de).}

\begin{document}

\maketitle
\begin{abstract}
    In this study, we explore a mathematical model for tissue growth focusing on the interplay between multiple cell subpopulations with distinct phenotypic characteristics. The model addresses the dynamics of tissue growth influenced by phenotype-dependent growth rates and collective population pressure, governed by Brinkman's law. We examine two primary objectives: the joint limit where viscosity tends to zero while the number of species approaches infinity, yielding an inviscid Darcy-type model with a continuous phenotype variable, and the continuous phenotype limit where the number of species becomes infinite with a fixed viscosity, resulting in a novel viscoelastic tissue growth model. In this sense, this paper  provides a comprehensive framework that elucidates the relationships between four different modelling paradigms in tissue growth.
    \\[0.5em]
\end{abstract}{}

\vskip .4cm
\begin{flushleft}
    \noindent{\makebox[1in]\hrulefill}
\end{flushleft}
\begin{flushright}
    \noindent{\makebox[1in]\hrulefill}
\end{flushright}
\vskip 1.5cm

\section{Introduction}

In the study of tissue growth and multi-cell aggregates, understanding the interplay between different species or phenotypes as well as their influence on the collective behaviour is crucial. 
A key aspect of this interplay involves how various subpopulations, each with distinct characteristics, respond to and, concurrently, influence the overall population density as well as  the overall growth dynamics within a given tissue. 

Mathematical models play an important role in unraveling these complex interactions and predicting the behaviour of biological systems under various conditions. \emph{Structured models}, which account for the heterogeneity within cell populations, provide a more accurate representation of biological systems. The resulting models can capture the dynamics of different cell types, their interactions, and their responses to environmental stimuli such as species-dependent growth rates. Moreover, structured models help shape the way we design tumour therapy \cite{LCC2016, LLCEP2015, CLC2016} and are a fundamental theoretical tool in the study of the development of drug resistance \cite{PBZMMJH2013, ABFHL2019}. 

The fundamental building block in phenotype-structured models often comprises systems of Lotka-Volterra type of the form
\begin{align*}
	\partialt n(t; a) - \Delta_a n(t; a) = R(n; a),
\end{align*}
where the density of cells, $n=n(t; a)$, is assumed to be well-mixed in space (spatial homogeneity) and labelled by a continuous phenotype variable $a$ \cite{LMP2011, Per2006, Per2015}, and references therein. In this equation, the linear diffusion term models mutation and the phenotype-dependent right-hand side models selection and growth dynamics of each subpopulation. Multiple variations of this model exist including spatial resolution \cite{LMV2020, LVLC2018} in form of linear random dispersal with trait-dependent mobility (by incorporating  ``$D(a) \Delta n$'') \cite{LMP2011} or by a nonlinear response to the pressure, $p$, within the tissue (by incorporating ``$D(a) \nabla \cdot (n(x,t;a) \nabla p )$''), \cite{LMV2020, MRL2022}.\\

In recent years, there has been significant progress in comprehending fluid-based models of tissue growth with variations of the continuity equation
\begin{align}
	\label{eq:continuity-eqn}
	\partialt n(x,t) + \nabla\cdot (n(x,t) v(x,t)) = f(x,t; n),
\end{align}
at their core \cite{Per2014}, and references therein. In this context, the population density, $n=n(x,t)$, at a point $x$ and time $t$, is influenced by a velocity field and subject to certain growth dynamics. Then, tissue growth arises from the interplay of these two effects: cell division driven by growth dynamics, $f$, increases the pressure, which, subsequently, accelerates dispersal as modelled by the velocity field $v$. Predominantly, there are three rheological laws in the literature relating the velocity and the pressure, $p$ --- \emph{Darcy's law}, \emph{Brinkman's law}, and \emph{Navier-Stokes equation}. In alignment with the focus of our paper, let us briefly review the literature on elastic models (Darcy coupling) and viscoelastic models (Brinkman coupling).\\

\textbf{Darcy's Law} \\
One of the simplest relationships between velocity and pressure is encapsulated in Darcy's law, $v = - \nabla p(n)$. Under this law, Eq. \eqref{eq:continuity-eqn} transforms into a semi-linear porous medium equation \cite{Vaz2007}, which was initially proposed in the context of tumour growth in \cite{BD2009} and later studied analytically in \cite{PQV2014} for the constitutive law $p(n) = n^\gamma$ and in \cite{HV2017} for the constitutive law $p(n) = \epsilon n / (1-n)$.
Historically, two-species variants of this Darcy model were introduced in \cite{BT1983, BGHP1985, BHIM2012, BHIHMW2020}, featuring a \emph{joint population pressure} generated by the presence of all cells, regardless of their type. More recently, under less restrictive conditions on the initial data and growth dynamics, this system has been studied in \cite{CFSS2017, BPPS2019, GPS2019, PX2020, Jac2021}, with a formal derivation of an $N$-species model presented in \cite{CLM2020}. Finally, \cite{Dav2023} considered a Darcy model with a continuous phenotype variable. However, a rigorous connection between models with a finite number of phenotypes $N < \infty$ and those with a continuous trait has, to the best of our knowledge, never been shown, and this will be addressed as part of this work.\\

\textbf{Brinkman's Law} \\
When viscoelastic effects are incorporated, the velocity in Eq. \eqref{eq:continuity-eqn} is related to the pressure via Brinkman's law, i.e., $v = - \nabla W$, where
$$
	- \nu \Delta W + W = p,
$$
with $\nu >0$ being the viscosity constant. This model was proposed and studied in \cite{PV2015, KT2018}.
In \cite{DT2015}, the authors considered a three-compartment model for proliferative cells, quiescent cells, and dead cells, with Brinkman coupling. These equations are coupled with a linear diffusion equation for nutrient distribution and drug distribution, respectively. Finally, a two-species version featuring sharp interfaces between the two phenotypes was proposed and studied analytically in \cite{DS2020, DPSV2021}.\\

\textbf{Inviscid Limit} \\
Formally, it can be observed that as $\nu \to 0$, $W \to p$. In this case, we recover Darcy's law in the inviscid limit. This limit was rigorously established in \cite{DDMS2024} for the constitutive law $p(n)=n$, and using the same technique, the inviscid limit was extended to the power law case $p(n) = n^\gamma$ in \cite{ES2023}, and to a certain class of systems involving interacting species in \cite{JVZ2024}.\\

\textbf{Phenotype Limit} \\
Models featuring $N \in \mathbb{N}$ distinct phenotypes quickly become numerically intractable. Moreover, in cases such as tumours, the diversity of phenotypes is so vast that it is often more effective to represent the phenotype space as a continuum, as mentioned earlier. This approach not only reduces the complexity of the mathematical model but also maintains biological accuracy.\\

\textbf{Goal of this work} \\
As starting point of this work, we consider an $N$-species system where $n^{(i)} = n^{(i)}(x,t)$ denotes the density of the $i$th subpopulation at location $x\in \Rd$ and time $t\geq 0$, where $i\in \set{1,\ldots,N}$. Each species responds to the collective population pressure via Brinkman's law, and their growth dynamics are governed by phenotype-dependent growth rates:
\begin{align*}
	\partialt {n^{(i)}} = \nabla \cdot (n^{(i)} \nabla W) + n^{(i)} G^{(i)}\prt*{(n^{(j)})_{j=1}^N},
\end{align*}
coupled through Brinkman's law
$$
	-\nu \Delta W + W = \frac1N\sum_{j=1}^N n^{(j)},
$$
and equipped with nonnegative initial data $n^{(i)}(x,0) = n^{(i),\text{in}}(x)$, for $i=1, \ldots, N$.\\

In this paper, we study this stratified tissue growth model, focusing on the behaviour of multiple species, $(n^{(i)})_{i=1}^N$, under different scaling limits. Specifically, the primary objectives of this paper are twofold:\\
\begin{wrapfigure}{r}{0.38\textwidth}
	\centering
	\vspace{1.\baselineskip}
	\includegraphics[width=0.4\textwidth]{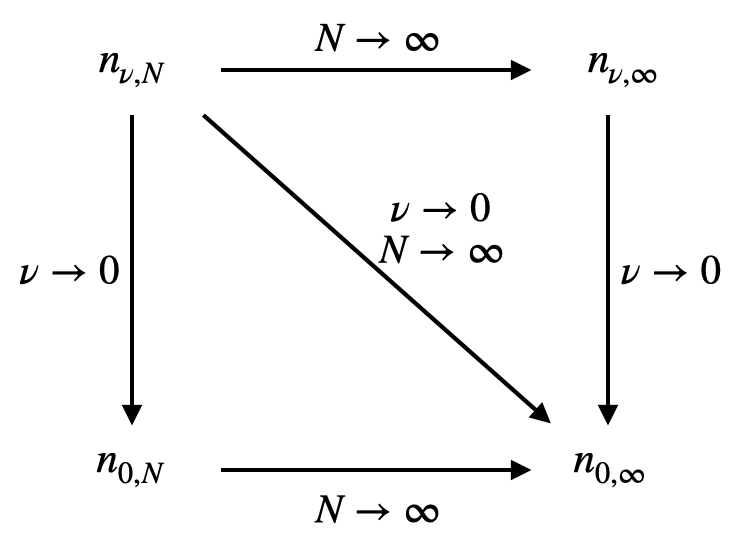}
	\vspace{-1.2\baselineskip}
\end{wrapfigure} 
\begin{enumerate}
	\item \textbf{Joint Limit:} We study the transition from a nonlocal, viscoelastic to an inviscid, local response in the tissue ($\nu \to 0$) while, simultaneously, let $N\to \infty$, to obtain an inviscid Darcy-type model with continuous phenotype variable.\\
	\item \textbf{Continuous Phenotype Limit:} Keeping the viscosity coefficient $\nu>0$ fixed, we let the number of distinct species tend to infinity. This limit provides a novel viscoelastic tissue growth model with continuous phenotype variable.\\
\end{enumerate}
The paper provides several novel regularity results which are stable under these different limits. Specifically, we establish a continuous-phenotype entropy inequality and we introduce the notion of weak solutions and derive entropy inequalities that play a crucial role in ensuring the stability and convergence of the solutions.

The rest of this paper is organised as follows. In Section \ref{sec:setting-main-result} we introduce the precise assumptions on the growth rates and the initial data. We introduce the notion of solutions, and present the main theorems. Subsequently, in Section \ref{sec:weak-sol-and-entropy}, we establish the existence of weak solutions to the $N$-species viscoelastic model and we prove several uniform estimates culminating in an entropy dissipation inequality which is a centrepiece in the joint limit. Section \ref{sec:JointLimit} is dedicated to the joint limit $\nu \to 0$ and $N\to \infty$. In Section \ref{sec:continuous-phenotype} we establish the continuous-phenotype limit, $N\to \infty$, while keeping $\nu > 0$ fixed and when $\nu =0$. We conclude in Section \ref{sec:conclusion} with some remarks and future avenues.

\section{Notation, definitions, and main results}
\label{sec:setting-main-result}
This section establishes the notation used throughout this work. We introduce the main assumptions regarding the initial data and growth rates of each species. Additionally, we define our concept of weak solutions for the respective systems. Finally, we present the main results of this paper.

\subsection{Notation}
Throughout, let $N\in \N$, with $N\geq 2$ denote the total number of distinct species. \\

\textbf{Growth rates.} We begin by addressing the assumptions on the growth rates. In the continuous-structure limit, we expect a growth rate function parameterised by the phenotype variable. Therefore, we consider 
\begin{align*}
	G: \R \times [0,1] &\to \R,\\
	(n, a) &\mapsto G(n; a),
\end{align*} 
and we define
\begin{equation*}
	G_N\i(n) \coloneqq G\prt{n; iN^{-1}},
\end{equation*}
for $i=1, \ldots, N$ and $n \geq 0$, with $C^1\cap L^\infty$-continuation on the negative half line. We make the following assumptions on $G$\\
\begin{itemize}
	\setlength\itemsep{1.em}
	\item[({\crtcrossreflabel{G1}[hyp:G1]})] regularity: $G \in C^{1} (\R \times [0,1])$,
	\item[({\crtcrossreflabel{G2}[hyp:G2]})] monotonicity: $ \max_{a \in [0,1]} \partial_{n} G(\cdot\,; a) \leq - \alpha <0 $ for some
	      $ \alpha > 0 $,
	\item[({\crtcrossreflabel{G3}[hyp:G3]})] homeostatic pressure: $\forall a \in [0,1] $ there is $ n^{\star}(a) >0 $ such that $ G(n^{\star}(a); a) =0$ and $\bar{n}:=\sup_{a \in [0,1]} n^{\star}(a) < \infty $.\\[1em]
\end{itemize}

\textbf{Initial data.} We define the initial data in a similar fashion. Let 
\begin{align*}
	n^{\mathrm{in}}: \Rd \times [0,1] &\to [0,\infty),\\
	(x, a) &\mapsto n^{\mathrm{in}}(x;a),
\end{align*}
be a Carath\'eodory function such that for all $a\in[0,1]$
\begin{equation}
	\label{eq:DataLpBound}
	n^{\mathrm{in}}(\cdot\,; a) \in L^{1}\cap L^{\infty}(\Rd),
\end{equation}
as well as
\begin{equation*}
	n^{\mathrm{in}}(x; a) \leq \bar{n},
\end{equation*}
for all $a\in[0,1]$ and almost every $x \in \Rd$.
We moreover assume that 
\begin{equation*}
    \sup_{a\in[0,1]}\int_{\Rd} n^{\mathrm{in}}(x; a)|x|^2 \dx x < \infty.
\end{equation*}

Then, we set
\begin{equation*}
	n^{(i),\mathrm{in}}_{N}(x) \coloneqq n^{\mathrm{in}}\left(x; iN^{-1} \right),
\end{equation*}
almost everywhere on $\Rd$.\\

\textbf{Starting point and further notation.} Having introduced all the necessary notation, we are now ready to present the starting point of our endeavour --  an $N$-species system, where $n_{\nu,N}^{(i)}=n_{\nu,N}^{(i)}(x,t)$ denotes the number density of the $i$th subpopulation at location $x$ and time $t$, for $i=1, \ldots, N$. Each species responds to the collective population pressure via Brinkman's law to avoid overcrowding, and the growth dynamics are governed by the phenotype-dependent growth rates introduced earlier. Altogether, the system's dynamics are expressed as follows:
{\begin{align}
	\label{eq:Brinkman_i}
	\left\{
	\begin{array}{rll}
		\ds \partialt{n\i\nun} - \div\prt*{n\i\nun \grad W\nun} \!\!\! & = \ds n\i\nun G\i_{N}(\underline{n}\nun ), & \\[0.5em]
		n\i\nun(x,0)                                                   & = n^{(i),\mathrm{in}}_{N}(x),              &                                  \\[0.5em]
		\ds -\nu \Lap W\nun +W\nun \!\!\!                              & = \ds  \underline{n}\nun,                  &
	\end{array}
	\right.
\end{align}
where the \emph{(rescaled) total population density} is defined as
\begin{equation*}
	\underline{n}\nun(x,t) = \frac{1}{N}\sum_{i=1}^{N} n\nun\i(x,t).
\end{equation*}
Let us introduce the \emph{interpolated density} as
\begin{equation} \label{def:interpotation}
	\begin{aligned}
		n\nun(x,t; a) \coloneqq  \sum_{i=1}^{N} n\i\nun(x,t) \indicator_{\left(\frac{i-1}{N},\frac{i}{N}\right]}(a),
	\end{aligned}
\end{equation}
along with the interpolated initial data
\begin{equation*}
	\begin{aligned}
		n_N^{\mathrm{in}}(x; a) \coloneqq  \sum_{i=1}^{N} n^{(i),\mathrm{in}}_N(x) \indicator_{\left(\frac{i-1}{N},\frac{i}{N}\right]}(a).
	\end{aligned}
\end{equation*}
Similarly, we introduce the interpolated growth rate:
\begin{equation*}
	\begin{aligned}
		G_{N}(n;a) \coloneqq \sum_{i=1}^{N} G_{N}\i(n)\indicator_{\left(\frac{i-1}{N},\frac{i}{N}\right]}(a).
	\end{aligned}
\end{equation*}
The advantage of introducing the interpolated quantities is that they naturally embed into function spaces continuous in space, time, and phenotype that play a crucial role in the continuous-phenotype limit. Indeed, using the interpolated quantities, System \eqref{eq:Brinkman_i} can be expressed as
{\begin{align}
	\label{eq:Brinkman_a}
	\left\{
	\begin{array}{rll}
		\ds \partialt{n\nun} - \div\prt*{n\nun \grad W\nun} \!\!\! & = \ds n\nun G_{N}(\underline{n}\nun; a), &  \\[0.5em]
		n\nun(x,0;a)                                               & = n^{\mathrm{in}}_{N}(x;a),               &                             \\[0.5em]
		\ds -\nu \Lap W\nun +W\nun \!\!\!                          & = \ds  \underline{n}\nun,                 &
	\end{array}
	\right.
\end{align}
for all $a \in [0,1]$. Upon observing that
\begin{equation*}
	\begin{aligned}
		\underline{n}\nun = \frac{1}{N}\sum_{i=1}^{N} n\i\nun = \sum_{i=1}^{N} n\nun\i\int_{i/N}^{i+1/N} 1 \dx{a} = \int_{0}^{1} n\nun(x,t;a)\dx{a},
	\end{aligned}
\end{equation*}
we find that the rescaled total population density satisfies
{\begin{align}
	\label{eq:Brinkman_average}
	\left\{
	\begin{array}{rll}
		\ds \partialt{\underline{n}\nun} - \div\prt*{\underline{n}\nun \grad W\nun} \!\!\! & = \ds \int_{0}^{1}  n\nun G_{N}(\underline{n}\nun;a)\dx{a}, & \\[0.5em]
		\underline{n}\i\nun(x,0)                                                           & = \underline{n}_N^{\mathrm{in}}(x),                           & \\[0.5em]
		\ds -\nu \Lap W\nun +W\nun \!\!\!                                                  & = \ds  \underline{n}\nun,                                     &
	\end{array}
	\right.
\end{align}
with
\begin{equation*}
    \underline{n}^{\mathrm{in}}_N := \int_0^1 n_N^{\mathrm{in}} \dx a.
\end{equation*}

\textbf{Limiting systems.} Based on System~\eqref{eq:Brinkman_a} we are interested in two limits --- the continuous phenotype limit ($N\to \infty$) and the inviscid limit ($\nu \to 0$). Regarding the first, letting $N\to \infty$, and assuming all limit objects exists, we formally obtain a phenotypically stratified viscoelastic tissue growth model of the form
{\begin{align}
	\label{eq:Brinkman_Phenotype}
	\left\{
	\begin{array}{rll}
		\ds \partialt{n_{\nu,\infty}} - \div\prt*{n_{\nu,\infty} \grad W_{\nu,\infty}} \!\!\! & = \ds n_{\nu,\infty} G(\underline{n}_{\nu,\infty};a),   \\[0.5em]
		n_{\nu,\infty}(x,0;a)                                                                 & = n^{\mathrm{in}}(x;a),                                 & \\[0.5em]
		\ds -\nu \Lap W_{\nu,\infty} +W_{\nu,\infty} \!\!\!                                   & = \ds  \underline{n}_{\nu,\infty},                      &
	\end{array}
	\right.
\end{align}
where
\begin{equation*}
	\underline{n}_{\nu,\infty} := \int_0^1 n_{\nu,\infty} \dx a.
\end{equation*}
Conversely, letting $\nu \to 0$, we obtain the $N$-species inviscid system, similarly to~\cite{DDMS2024},
\begin{align}
	\label{eq:Darcy_a}
	\left\{
	\begin{array}{rll}
		\ds \partialt{n_{0,N}} - \div\prt*{n_{0,N} \grad \underline{n}_{0,N}} \!\!\! & = \ds n_{0,N} G_{N}(n_{0,N};a), &  \\[0.5em]
		n\nun(x,0;a)                                               & = n^{\mathrm{in}}_{N}(x;a),               &                             \\[0.5em]
	\end{array}
	\right.
\end{align}
where
\begin{equation*}
	\underline{n}_{0, N} := \int_0^1 n_{0, N} \dx a.    
\end{equation*}

However, the second result of this work is to obtain compactness sufficient to pass to the joint limit and obtain the phenotypically stratified inviscid system
{\begin{align}
	\label{eq:Darcy_Phenotype}
	\left\{
	\begin{array}{rll}
		\ds \partialt{n_{0,\infty}} - \div\prt*{n_{0,\infty} \grad \underline{n}_{0,\infty}} \!\!\! & = \ds n_{0,\infty} G(\underline{n}_{0,\infty};a),   \\[0.5em]
		n_{0,\infty}(x,0;a)                                                                 & = n^{\mathrm{in}}(x;a),                                 & \\[0.5em]
	\end{array}
	\right.
\end{align}
where, as before,
\begin{equation*}
	\underline{n}_{0,\infty} := \int_0^1 n_{0, \infty} \dx a.
\end{equation*}

Below we introduce the notion of weak solutions that we will be working with in this paper.
\begin{definition}[Weak Solutions - Brinkman]
	\label{def:wk_sol_brinkman}
	We say that the pair $(n\nun,W\nun) $ is a \emph{weak solution} to System \eqref{eq:Brinkman_a} with nonnegative initial data $n\nun^{\mathrm{in}} \in L^1(\Rd)\cap L^\infty(\Rd)$ if, for almost every $a \in [0,1]$, $n\nun(\cdot\ ; a) \in L^\infty(0,T; L^1(\Rd) \cap L^\infty(\Rd))$ is nonnegative and there holds
	\begin{align*}
		\int_0^T\!\!\!  \int_\Rd & n\nun \partialt \varphi \dx x \dx t
		- \int_0^T \!\!\! \int_\Rd n\nun \nabla W\nun \cdot \nabla \varphi \dx x \dx t
		\\
		& = -\int_0^T \!\!\! \int_\Rd \varphi n\nun G_{N}(\underline{n}\nun; a) \dx x \dx t -\int_\Rd \varphi(x,0)n^{\mathrm{in}}\nun(x; a) \dx x ,
	\end{align*}
	for almost every $a\in[0,1]$ and any test function $\varphi \in C_{c}^{\infty}(\Rd \times [0,T))$, as well as
	\begin{align*}
		- \nu \Lap W\nun + W\nun = \underline{n}\nun,
	\end{align*}
	almost everywhere in $\Rd\times(0,T)$.\\ 
    When $N=\infty$, the same properties define a weak solution to System~\eqref{eq:Brinkman_Phenotype} (with the convention that $G_\infty = G$).
\end{definition}

At this stage, let us recall that any solution of Brinkman's equation can be expressed as the convolution with the fundamental solution, denoted by $K_\nu$, \ie, the solution of
\begin{align*}
	- \nu \Delta W\nun + W\nun = \underline{n}\nun,
\end{align*}
can be represented as
\begin{align*}
	W\nun = K_{\nu} \star \underline{n}\nun,
\end{align*}
where
\begin{equation*}
	K_{\nu}(x) = \frac{1}{4\pi}\int_0^\infty \exp\prt*{-\frac{\pi|x|^2}{4s\nu} - \frac{4s}{\pi}}s^{-d/2}\dx s.
\end{equation*}
In particular $K_\nu \geq 0$ and $\int K_\nu = 1$.\\

Regarding the inviscid limit, let us now introduce the notion of weak solution to the limiting Darcy models, System \eqref{eq:Darcy_a} and System~\eqref{eq:Darcy_Phenotype}.
\begin{definition}[Weak Solutions - Darcy]
	\label{def:wk_sol_darcy}
	We call  $n_{0,N} \geq 0$ a \emph{weak solution} to System \eqref{eq:Darcy_a} with nonnegative initial data $n_{0,N}^{\mathrm{in}} \in L^1(\Rd)\cap L^\infty(\Rd)$ if, for almost every $a\in [0,1]$,
    \begin{equation*}
        n_{0,N}(\cdot\,; a) \in L^\infty(0,T; L^1(\Rd)\cap L^\infty(\Rd))\quad\text{and}\quad \underline{n}_{0,N} \in L^2(0,T; H^1(\Rd)),
    \end{equation*}
 and there holds
	\begin{equation*}
		\begin{aligned}
			\int_0^T \!\!\! \int_\Rd
			 & n_{0,N} \partialt \varphi \dx x \dx t
			- \int_0^T \!\!\! \int_\Rd n_{0,N} \nabla \underline{n}_{0,N} \cdot \nabla \varphi \dx x \dx t                                                            \\
			 & =- \int_0^T \!\!\! \int_\Rd \varphi n_{0,N} G_{N}(\underline{n}_{0,N}; a) \dx x \dx t - \int_\Rd \varphi(x,0) n_{0,N}^{\mathrm{in}}(x; a) \dx x,
		\end{aligned}
	\end{equation*}
	for almost every $a\in[0,1]$ and any test function $\varphi \in C_{c}^{\infty}(\Rd\times[0,T))$.\\
 When $N=\infty$, the same properties define a weak solution to System~\eqref{eq:Darcy_Phenotype} (with the convention that $G_\infty = G$).
\end{definition}

\subsection{Main results}
Our first result concerns the existence of weak solutions to the $N$-species viscoelastic system, System \eqref{eq:Brinkman_i} as well as establishing an entropy inequality satisfied by the rescaled total population density.
\begin{lemma}[Existence of weak solutions]
	\label{lem:Existence}
	There exists a weak solution $ (n\i\nun,W\nun)_{i=1}^N $ of System~\eqref{eq:Brinkman_i} such that\\
	\begin{enumerate}[(i)]
		\setlength\itemsep{1.em}
		\item $ n\i\nun \in L^{\infty}(0,T;L^{1}\cap L^{\infty}(\Rd))$, uniformly in $\nu$ and $N$,
		\item $\exists C>0$ such that $0 \leq n\i\nun \leq  C $, uniformly in $\nu$ and $N$,
		\item $ W\nun \in L^{\infty}(0,T;L^{1}\cap L^{\infty}(\Rd))$, uniformly in $\nu$ and $N$,
		\item $W\nun \in L^{\infty}(0,T;W^{1,q}(\Rd)) \cap L^{\infty}(0,T;W^{2,r}(\Rd)) $,  for $ 1\leq q \leq \infty $ and $ 1 < r < \infty $, uniformly in $N$,\\
	\end{enumerate}
	and such that $ n\nun $ as defined in Eq.~\eqref{def:interpotation} satisfies System~\eqref{eq:Brinkman_a} in the sense of Definition~\ref{def:wk_sol_brinkman}.
\end{lemma}
\begin{lemma}[Entropy inequality]
	\label{lem:entropy-ineq}
	Let $ (n\i\nun,W\nun)_{i=1}^N $ be the weak solution constructed in Lemma \ref{lem:Existence}. Then, the following entropy inequality holds:
	\begin{align*}
		\curlyH[\underline{n}\nun](T) - \curlyH[\underline{n}\nun^{\mathrm{in}}]
		& - \int_0^T \!\!\! \int_\Rd  \underline{n}\nun \Lap W\nun  \dx{x}\dx{t} \\
		& \leq\int_0^T \!\!\! \int_\Rd \log \underline{n}\nun \int_{0}^{1} n\nun G_{N}(\underline{n}\nun;a)\dx{a} \dx x \dx t,
      \end{align*}
      where
      \begin{equation*}
	      \curlyH[f](t) \coloneqq \int_\Rd f(x, t)(\log f(x, t) - 1) \dx x.
      \end{equation*}
\end{lemma}

\begin{theorem}[Continuous phenotype limit]
\label{thm:PhenotypeLimit}
	Let $\nu>0$ be fixed. For $N\in \N$, let $(n\i\nun, W\nun)_{i=1}^N $ be the weak solution constructed in Lemma \ref{lem:Existence}. Then, there exists a function $n_{\nu, \infty}: \Rd \times [0,T] \times [0,1] \to [0, \infty)$ with
    \begin{equation*}
        n_{\nu, \infty}(\cdot, \cdot, a)\in L^\infty(0,T; L^1(\Rd) \cap L^\infty(\Rd)),\quad\text{for almost every $a \in [0,1]$},
    \end{equation*}
    such that, up to a subsequence, 
    \begin{align*}
         n_{\nu, N} &\stackrel {\star}{\rightharpoonup} n_{\nu, \infty} \quad\text{in } L^\infty(0,T;L^p(\Rd)), 1\leq p\leq \infty,\\
        \underline{n}_{\nu, N} &\to \underline{n}_{\nu,\infty} \quad\text{in } L^2(0,T;L^2(\Rd)),\\
        W_{\nu, N} &\to W_{\nu,\infty}:=K_\nu\star\underline{n}_{\nu,\infty}  \quad\text{in } L^2(0,T;H^1_{\mathrm{loc}}(\Rd)).
    \end{align*}
    The limit $(n_{\nu,\infty}, W_{\nu,\infty})$ is a weak solution to System \eqref{eq:Brinkman_Phenotype} in the sense of Definition~\ref{def:wk_sol_brinkman}.
\end{theorem}

\begin{theorem}[Joint limit]
\label{thm:JointLimit}
	For $N\in \N$ and $\nu >0$, let $ (n\i\nun,W\nun)_{i=1}^N $ be the weak solution constructed in Lemma \ref{lem:Existence}. Then, there exists a function $n_{0, \infty}: \Rd \times [0,T] \times [0,1] \to [0, \infty)$ with 
    \begin{equation*}
        n_{0, \infty}(\cdot, \cdot\,; a)\in L^\infty(0,T; L^1(\Rd) \cap L^\infty(\Rd)),\quad \text{for almost every $a \in [0,1]$},
    \end{equation*}
    and
    \begin{equation*}
        \underline{n}_{0, \infty} \in L^2(0,T;H^1(\Rd)),
    \end{equation*}
    such that, up to a subsequence, 
    \begin{align*}
        n_{\nu, N} &\stackrel {\star}{\rightharpoonup} n_{0,\infty} \quad\text{in } L^\infty(0,T;L^p(\Rd)), 1\leq p\leq \infty,\\
        \underline{n}_{\nu, N} &\to \underline{n}_{0,\infty} \quad\text{in } L^2(0,T;L^2(\Rd)),\\
        W_{\nu, N} &\to \underline{n}_{0,\infty} \quad\text{in } L^2(0,T;H^1(\Rd)).
    \end{align*}
    The limit function $n_{0,\infty}$ is a weak solution to System~\eqref{eq:Darcy_Phenotype} in the sense of Definition \ref{def:wk_sol_darcy}.
\end{theorem}

\section{Weak solutions and the entropy inequality}
\label{sec:weak-sol-and-entropy}
This section is devoted to proving Lemma~\ref{lem:Existence} and Lemma \ref{lem:entropy-ineq}. To this end, we first introduce a parabolic regularisation and, in a series of auxiliary propositions, derive the desired uniform bounds. Then, we prove that the solutions to the regularised system converge strongly to weak solutions of System~\eqref{eq:Brinkman_i}. Finally, we establish the entropy inequality.

\subsection{Approximate system}
Consider the following parabolic regularisation of System~\eqref{eq:Brinkman_i}
{\begin{align}
	\label{eq:regularised_Brinkman_i}
	\left\{
	\begin{array}{rll}
		\ds \partialt{n\i\epsnun} - \epsilon \Lap n\i\epsnun  \!\!\! & = \div\prt*{n\i\epsnun \grad W\epsnun}  + \ds n\i\epsnun G\i_{N}(\underline{n}\epsnun ), & i = 1, \dots, N, \\[0.5em]
		n\i\epsnun(x,0)                                              & = n^{(i),\mathrm{in}}_{\epsilon,N}(x),                                                   &                 \\[0.5em]
		\ds -\nu \Lap W\epsnun +W\epsnun \!\!\!                      & = \ds  \underline{n}\epsnun.                                                       &
	\end{array}
	\right.
\end{align}
As in the previous section, we introduce the piecewise constant interpolation in phenotype
\begin{equation*}
	n_{\epsilon,\nu,N}(x,t;a) := \sum_{i=1}^{N} n\i_{\epsilon,\nu,N}(x,t) \indicator_{\left(\frac{i-1}{N},\frac{i}{N}\right]}(a),
\end{equation*}
and rewrite System \eqref{eq:regularised_Brinkman_i} as 
{\begin{align}
	\label{eq:regularized_Brinkman_a}
	\left\{
	\begin{array}{rll}
		\ds \partialt{n\epsnun}- \epsilon \Lap n\epsnun\!\!\!
		                                        & =  \div\prt*{n\epsnun \grad W\epsnun}  + \ds n\epsnun G_{N}(\underline{n}\epsnun; a), & \forall  a \in [0,1], \\[0.5em]
		n\epsnun(x,0;a)                         & = n^{\mathrm{in}}_{\epsilon,N}(x;a),                                                   &                      \\[0.5em]
		\ds -\nu \Lap W\epsnun +W\epsnun \!\!\! & = \ds \underline{n}\epsnun.                                                            &
	\end{array}
	\right.
\end{align}

\textbf{Regularised initial data.} The initial data for the above regularised system is given as follows. Let $n^{\mathrm{in}}_\epsilon \in C^\infty_c(\R^d \times [0,1])$ approximate $n^{\mathrm{in}}$ such that
\begin{equation*}
	\norm{n_\epsilon^{\mathrm{in}}(\cdot\,; a)}_{L^p(\R^d)} \lesssim \norm{n^{\mathrm{in}}(\cdot\,; a)}_{L^p(\R^d)},
\end{equation*}
for any $1\leq p\leq \infty$ and all $a\in[0,1]$, and
$$
	n_\epsilon^{\mathrm{in}}(\cdot\,; a) \longrightarrow n^{\mathrm{in}}(\cdot\,; a),
$$ 
strongly in $L^1(\Rd)$ for each $a\in[0,1]$ as well as almost everywhere in $\R^d \times [0,1]$. Finally, $n^{(i),\mathrm{in}}_{\epsilon,N}$, $n^{\mathrm{in}}_{\epsilon,N}$, and $\underline{n}_{\epsilon,N}^{\mathrm{in}}$ are defined in the same way as in the previous section.
Let us point out that the condition
\begin{equation*}
    \int_{\Rd} n^{\mathrm{in}}|x|^2 \dx x < \infty
\end{equation*}
can be guaranteed to be preserved by the approximation of the initial data, i.e., we have
\begin{equation}
\label{eq: 2nd moment data}
    \sup_{\epsilon>0}\,\int_{\Rd} n_\epsilon^{\mathrm{in}}|x|^2 \dx x < \infty.
\end{equation}
Indeed, let $R>0$ be fixed. Then, for $\epsilon$ small enough we have
\begin{align*}
    \int_{\Rd} n_\epsilon^{\mathrm{in}}|x|^2\indicator_{B_R(0)} \dx x \leq 1 + \int_{\Rd} n^{\mathrm{in}}|x|^2 \dx x.
\end{align*}
Passing to the limit $R\to\infty$, we obtain~\eqref{eq: 2nd moment data}.

\textbf{Existence of solutions.}
Following the construction of \cite{ES2023}, we obtain a nonnegative solution $n\i_{\epsilon,\nu,N}$ which is bounded uniformly in $\epsilon$ in the following sense
\begin{align*}
	n\i_{\epsilon,\nu,N}\in L^\infty(0,T;L^1\cap L^\infty(\R^d)),\quad
	\partial_t n\i_{\epsilon,\nu,N} & \in L^2(0,T;H^{-1}(\R^d)).
\end{align*}
Moreover, the following regularity holds for each $\epsilon>0$
\begin{align*}
	n\i_{\epsilon,\nu,N}           \in L^2(0,T;H^1(\R^d)),\quad
	\partial_t n\i_{\epsilon,\nu,N}  \in L^2(0,T;L^2(\R^d)).
\end{align*}

\subsection{Uniform estimates in $\nu$ and $N$}

Before proving strong convergence of the regularised densities towards a weak solution of the Brinkman system, let us discuss some properties of the regularised rescaled total population density, $\underline{n}\epsnun$, which satisfies the equation
\begin{align}
	\label{eq:regularised-equation}
	\left\{
	\begin{array}{rll}
		\ds \partialt {\underline{n}\epsnun} - \epsilon \Lap \underline{n}\epsnun \!\!\!
		 & = \nabla \cdot (\underline{n}\epsnun \nabla W\epsnun)                                                 \\[0.5em]
		 & \quad \ds + \int_{0}^{1}  n\epsnun G_{N}(\underline{n}\epsnun;a)\dx{a},
		 & \text{in } \Rd\times(0,T)\times[0,1],                                                                 \\[0.75em]
		\underline{n}\epsnun(x,0) \!\!\!
		 & = \ds\underline{n}_{\epsilon,N}^{\mathrm{in}}(x)\coloneqq  \int_{0}^{1}  n_{\epsilon,N}^{\mathrm{in}}(x;a)\dx{a}.
		 &
	\end{array}
	\right.
\end{align}

\begin{proposition}[$L^\infty$ for total population]
    There holds
    \begin{align*}
        0 \leq \underline n \epsnun \leq \bar n,
    \end{align*}
    almost everywhere.
\end{proposition}
\begin{proof}
    The uniform $L^\infty$-bound is obtained by a standard argument following~\cite{TangEtAl}. To this end, we assume $(x^\star, t^\star)$ is a maximum point for $\underline n\epsnun$ and $\underline n\epsnun(x^\star, t^\star) > \bar n$. In such a point, there holds
\begin{align*}
	0 = \partialt {\underline n\epsnun}(x^\star, t^\star) 
	&= \epsilon \Delta \underline n\epsnun(x^\star, t^\star) + \nabla \underline n\epsnun(x^\star, t^\star) \cdot \nabla W\epsnun(x^\star, t^\star) \\
	&+ \underline n\epsnun(x^\star, t^\star) \Delta W\epsnun(x^\star, t^\star) + \int_0^1 n\epsnun(x^\star, t^\star) G(\underline  n\epsnun(x^\star, t^\star); a) \dx a\\
	&< \nu^{-1} \underline n\epsnun(x^\star, t^\star)(W\epsnun(x^\star, t^\star) - \underline n\epsnun(x^\star, t^\star))\\
	&\leq 0,
\end{align*}
which is a contradiction. Hence, $0 \leq \underline n\epsnun \leq \bar n$.
\end{proof}

As an immediate consequence, let us integrate Equation~\eqref{eq:regularised_Brinkman_i} to obtain
\begin{equation*}
    \ddt \norm{n\i\epsnun}_{L^1(\Rd)} \leq \norm{G}_{L^\infty([0,\bar n] \times [0,1])}\norm{n\i\epsnun}_{L^1(\Rd)},
\end{equation*}
so that, by Gronwall's lemma, we deduce that $n\i\epsnun \in L^\infty(0,T;L^1(\Rd))$, uniformly in $\epsilon$, $\nu$, and $N$.
Using the uniform $L^\infty$-control on the rescaled total population we also get the following result.

\begin{proposition}[Uniform $L^\infty$-control]
    For $N\in\N$, $\nu>0$ and $\epsilon>0$, let $({n\i\epsnun})_{i=1}^N$ be the solution to Eq. \eqref{eq:regularised_Brinkman_i}. Then, there holds
    \begin{align}
        0 \leq n\i\epsnun \leq C, 
    \end{align}
    where the constant $C>0$ is independent of $N$, $\nu$, and $\epsilon$. Furthermore, there holds
    \begin{align*}
        0 \leq n\epsnun(x,t;a) \leq C,
    \end{align*}
    almost everywhere with the same constant $C$.
\end{proposition}
\begin{proof}
We define the quantity
\begin{equation*}
    \Psi(x, t) := \alpha  \underline n\epsnun (x, t) e^{2 \beta t},
\end{equation*}
where
\begin{equation*}
    \alpha = \frac{1}{\bar n} \left\| n_{\epsilon, N}^{\mathrm{in}} \right\|_{L^\infty(\mathbb{R}^d \times [0,1])}\quad \text{and}\quad\beta = \left\| G \right\|_{L^\infty([0,\bar n] \times [0,1])} .
\end{equation*}

Let $|u|_{-}$ and $\sign_{-}(u)$ we denote the negative part and negative sign of the function $u$, i.e.,
    \begin{equation*}
        |u|_{-} :=
        \begin{cases}
            -u, & \text{for $u<0$,}\\
            0, & \text{for $u\geq 0$,}
        \end{cases}
        \qquad \text{and}\qquad \sign_{-}(u):=
        \begin{cases}
            -1, & \text{for $u<0$,}\\
            0, & \text{for $u\geq 0$.}
        \end{cases}
    \end{equation*}
Then, we consider
\begin{align*}
    &\ddt \int_{\Rd} \left| \Psi(x, t) - n\epsnun(x, t; a) \right|_- \, \dx x
    = \int_{\Rd} \mathrm{sign}_- \left( \Psi - n\epsnun \right) \left( \partial_t \Psi - \partial_t n\epsnun \right) \dx x\\
    &= \int_{\Rd} \mathrm{sign}_-\left( \Psi - n\epsnun \right) \left( 2 \beta \Psi + \alpha e^{2 \beta t} \partial_t \underline n\epsnun - \partial_t n\epsnun \right) \dx x\\
    &= \int_{\Rd} \mathrm{sign}_-\left( \Psi - n\epsnun \right) \nabla \cdot \prt*{\left( \Psi - n\epsnun \right) \nabla W\epsnun} \dx x  \\
    &\, + \int_{\Rd} \mathrm{sign}_-\left( \Psi - n\epsnun \right) \left( 2 \beta \Psi + \alpha e^{2 \beta t} \int_0^1 n\epsnun G_N(\underline{n}\epsnun; \tilde a) d \tilde a - n\epsnun G_N(\underline{n}\epsnun;a) \right) \dx x\\
    & \leq \int_{\Rd} \mathrm{sign}_-\left( \Psi - n\epsnun \right) \left( \beta \Psi - n\epsnun G_N(\underline{n}\epsnun;a) \right) \dx x\\
    & \leq \beta \int_{\Rd} \left| \Psi - n\epsnun \right|_{-} \dx x.
\end{align*}

By Gronwall's Lemma, we obtain
\begin{equation*}
    0 \leq \sup_{t \in [0,T]} \sup_{a \in [0,1]} \int_{\Rd} \left| \Psi - n\epsnun \right|_- \dx x \leq 0,
\end{equation*}
and therefore
\begin{equation*}
    0 \leq n\epsnun(x, t; a) \leq \left\| n_{\epsilon, N}^{\mathrm{in}} \right\|_{L^\infty(\Rd)} e^{2 \left\| G \right\|_{L^\infty([0,\bar n] \times [0,1])}T}.
\end{equation*}

The result follows by observing that
\begin{equation*}
    \left\|n_{\epsilon, N}^{\mathrm{in}}\right\|_{L^\infty(\Rd)} \lesssim \left\|n^{\mathrm{in}}\right\|_{L^\infty(\Rd)}.
\end{equation*}
\end{proof}

Having established the uniform $L^p$-bounds, we derive further regularity results.
First, we prove a bound on the gradient of the solution
to Eq.~\eqref{eq:regularised-equation}.
\begin{proposition}
	\label{prop:regular_n_bounds}
	Let $(\underline{n}\epsnun)_{\epsilon>0}$ be the family of solutions to Eq.~\eqref{eq:regularised-equation}. Then, there holds
	\begin{align*}
		\sqrt{\epsilon} \norm*{\nabla \underline{n}\epsnun}_{L^2(0,T; L^2(\Rd))} \leq C,
	\end{align*}
	for some constant $C>0$ independent of $\epsilon$, $\nu$, and $N$.
\end{proposition}
\begin{proof}
	Testing Eq.~\eqref{eq:regularised-equation} by the solution and integrating by parts, we obtain
	\begin{equation}\label{eq:AuxiliaryGradient}
    \begin{aligned}
		\frac12 \dfrac{\dx}{\dx t} & \int_\Rd
		\underline{n}\epsnun^2  \dx x + \epsilon \int_\Rd \abs{\grad
		\underline{n}\epsnun}^{2} \dx x                                                                                                                                                               \\
		                           & =  \frac12 \int_\Rd
		\underline{n}\epsnun^2 \Lap W\epsnun \dx x +  \int_{\Rd}^{} \underline{n}\epsnun
		\prt*{\int_{0}^{1} n\epsnun G_{N}(\underline{n}\epsnun;a) \dx{a}} \dx{x}                                                                                                                      \\
		& \leq \frac12 \int_\Rd
		\underline{n}\epsnun^2 \Lap W\epsnun \dx x + \norm{G}_{L^\infty([0,\bar n] \times [0,1])}\int_{\Rd} \underline{n}\epsnun^2 \dx x.
	\end{aligned}
    \end{equation}
    We now observe that the first term on the right-hand side has a sign:
    \begin{equation}\label{eq:nDeltaW-sign}
    \begin{aligned}
        \int_\Rd\underline{n}\epsnun^2 \Lap W\epsnun \dx x &= \frac{1}{\nu}\int_\Rd\underline{n}\epsnun^2 \prt*{W\epsnun-\underline{n}\epsnun} \dx x \\
        &\leq  \frac{1}{\nu}\norm{\underline{n}\epsnun}_{L^3(\Rd)}^2\norm{W\epsnun}_{L^3(\Rd)} - \frac{1}{\nu}\int_\Rd\underline{n}\epsnun^3 \dx x\\
        &\leq \frac{1}{\nu}\norm{\underline{n}\epsnun}_{L^3(\Rd)}^3\norm{K_\nu}_{L^1(\Rd)} - \frac{1}{\nu}\int_\Rd\underline{n}\epsnun^3 \dx x\\
        &\leq 0,
    \end{aligned}
    \end{equation}
    recalling that $\norm{K_\nu}_{L^1(\Rd)}=1$.
	Using this computation in~\eqref{eq:AuxiliaryGradient} and applying Gronwall's lemma yields the statement.
\end{proof}

\begin{proposition}
	\label{lem:n_gradW2eps}
	Let $ \left(\underline{n}\epsnun \right)_{\epsilon > 0} $ be the family of solutions to Eq.~\eqref{eq:regularised-equation}. Then,
	\begin{equation*}
		\begin{aligned}
			\int_{0}^{T}\!\!\! \int_{\Rd}^{}  \underline n\epsnun \abs*{\grad W\epsnun}^{2} \dx{x}\dx{t}
			\leq C,
		\end{aligned}
	\end{equation*}
	where $C$ is independent of $ \epsilon $, $\nu$, and $N$.
\end{proposition}
\begin{proof}
	Multiplying Eq.~\eqref{eq:regularised-equation} by $W\epsnun$ and integrating in space, we obtain
	\begin{equation}
		\begin{aligned}
			\label{eq:2ndMoment_intermediate}
			\frac12  \ddt \int_\Rd W\epsnun\, \underline{n}\epsnun \dx x
			 & = \epsilon \int_\Rd \underline{n}\epsnun\Delta W\epsnun \dx x
			-\int_\Rd \underline{n}\epsnun \abs*{\nabla W\epsnun}^2 \dx x                                              \\
			 & \quad + \int_\Rd W\epsnun\left[ \int_{0}^{1} n\epsnun G_{N}(\underline{n}\epsnun;a)\dx{a}\right] \dx x.
		\end{aligned}
	\end{equation}
	For the reaction term, we have
	\begin{equation*}
		\int_\Rd W\epsnun \left[ \int_{0}^{1} n\epsnun G_{N}(\underline{n}\epsnun;a)\dx{a}\right] \dx x \leq \norm{G}_{L^\infty([0,\bar n] \times [0,1])}\int_{\Rd}^{} W\epsnun\, \underline{n}\epsnun \dx{x},
	\end{equation*}
    while proceeding similarly as in~\eqref{eq:nDeltaW-sign}, we see that
    \begin{equation*}
        \epsilon \int_\Rd \underline{n}\epsnun\Delta W\epsnun \dx x \leq 0.
    \end{equation*}
	Thus, we obtain
	\begin{equation*}
			\frac12  \ddt \int_\Rd W\epsnun\, \underline{n}\epsnun \dx x + \int_\Rd   \underline{n}\epsnun \abs*{\nabla W\epsnun}^{2}  \dx x                                                         
			 \leq C \int_{\Rd}^{} W\epsnun \underline{n}\epsnun \dx{x}.
	\end{equation*}
	The result follows from Gronwall's lemma.
\end{proof}

\begin{proposition}
\label{prop:TimeDerivative}
    Let $(\underline{n}\epsnun)_{\epsilon>0}$ be the family of solutions to Eq.~\eqref{eq:regularised-equation}. Then, there holds
    \begin{equation*}
        \partial_t \underline{n}\epsnun \in L^2(0,T;H^{-1}(\Rd)),
    \end{equation*}
    uniformly in $\epsilon$, $\nu$, and $N$.
\end{proposition}
\begin{proof}
    For a test function $\varphi \in L^2(0,T;H^1(\Rd))$, we have
    \begin{align*}
        \abs*{\int_0^T\!\!\!\int_{\Rd} \partial_t \underline{n}\epsnun \varphi \dx x\dx t} &\leq \epsilon \int_0^T\!\!\!\int_{\Rd} |\nabla \underline{n}\epsnun||\nabla\phi|\dx x \dx t + \int_0^T\!\!\!\int_{\Rd} \underline{n}\epsnun|\nabla W\epsnun||\nabla\phi|\dx x \dx t\\
        &\quad+ \norm{G}_{L^\infty([0,\bar n] \times [0,1])}\int_0^T\!\!\!\int_{\Rd} \underline{n}\epsnun|\phi|\dx x \dx t\\
        &\leq C\norm{\nabla\varphi}_{L^2(0,T;L^2(\Rd))} + C\norm{\varphi}_{L^2(0,T;L^2(\Rd))},
    \end{align*}
    where the constants are independent of any parameters by Propositions~\ref{lem:n_gradW2eps} and~\ref{prop:regular_n_bounds}.
\end{proof}

\begin{proposition}
	\label{lem:second_moment}
	Let $\underline{n}\epsnun$ be the solution to Eq.~\eqref{eq:regularised-equation}. Assume that there is a constant $C>0$ such that
	\begin{align*}
		\sup_{\epsilon >0} \int_\Rd \underline{n}_{\epsilon, N}^{\mathrm{in}}|x|^2 \dx x \leq C.
	\end{align*}
	Then, the second moment remains bounded and there holds
	\begin{align*}
		\sup_{t\in [0,T]}\int_\Rd \underline{n}\epsnun(x,t) \abs{x}^{2}  \dx x\leq C,
	\end{align*}
	for some constant independent of $\epsilon$, $\nu$, and $N$.
\end{proposition}
\begin{proof}
	We compute
	\begin{equation*}
		\begin{aligned}
			\frac{1}{2}  \ddt \int_\Rd \underline{n}\epsnun \abs{x}^{2} \dx x
			 & =  \epsilon d \int_\Rd \underline{n}\epsnun \dx x - \int_\Rd \underline{n}\epsnun x \cdot \nabla W\epsnun \dx x \\
			 & \quad +  \frac{1}{2} \int_\Rd \abs{x}^{2}  \prt*{\int_{0}^{1} n\epsnun  G_{N}(\underline{n}\epsnun;a)\dx{a}} \dx x \\[0.0em]
			 & \leq \epsilon d \norm{\underline{n}\epsnun}_{L^\infty(0,T; L^1(\Rd))} \\
			 & \quad+\frac12 \prt*{ 1 +  \sup_{a\in [0,1]}\norm{G}_{L^\infty([0,\bar n]\times[0,1])}}\int_\Rd \underline{n}\epsnun|x|^2\dx x \\
			 & \quad + \frac12 \int_\Rd \underline{n}\epsnun |\nabla W\epsnun|^2 \dx x                                                             \\[0.3em]
			 & \leq C + C \int_\Rd  \underline{n}\epsnun \abs{x}^{2} \dx x,
		\end{aligned}
	\end{equation*}
	and, again, we conclude by Gronwall's lemma.
\end{proof}

Let us point out that the above proposition implies also uniform second-moment control for $n\epsnun(x,t;a)$ for all phenotypes $a\in [0,1]$.

\begin{lemma}[Entropy Bounds]
	\label{lem:entropy-bounds}
	Let $\underline{n}\epsnun$ be the solution to Eq. \eqref{eq:regularised-equation}. Then, there holds
	\begin{align*}
		\sup_{t\in[0,T]} \int_\Rd \underline{n}\epsnun (x,t) \abs{\log \underline{n}\epsnun (x,t)} \dx x \leq C,
	\end{align*}
	for some constant independent of $\epsilon$, $\nu$, and $N$.
\end{lemma}

\begin{proof}
	Let us consider
	\begin{align*}
		\int_\Rd \underline{n}\epsnun \abs{\log \underline{n}\epsnun} \dx x
		 & = \int_{\set{\underline{n}\epsnun \geq 1}} \underline{n}\epsnun \log \underline{n}\epsnun \dx x - \int_{\set{\underline{n}\epsnun < 1}} \underline{n}\epsnun \log \underline{n}\epsnun \dx x \\[0.5em]
		 & \leq \norm{\underline{n}\epsnun}_{L^\infty(0,T; L^\infty(\Rd))} \norm{\underline{n}\epsnun}_{L^\infty(0,T;L^1(\Rd))}  + J,
	\end{align*}
	where
	\begin{align*}
		J \coloneqq - \int_{\set{\underline{n}\epsnun < 1}} \underline{n}\epsnun \log \underline{n}\epsnun \dx x.
	\end{align*}
	In order to estimate $J$, let $\mathcal{N}(x)$ denote the standard normal Gaussian. Then, we have
	\begin{align*}
		J & = - \int_{\set{\underline{n}\epsnun < 1}} \underline{n}\epsnun \log \underline{n}\epsnun \dx{x}                                                                                                                                                                                                                               \\
		  & = - \int_\Rd \frac{\underline{n}\epsnun \indicator_{\set{\underline{n}\epsnun < 1}}}{\mathcal{N}(x)} \log\prt*{\frac{\underline{n}\epsnun \indicator_{\set{\underline{n}\epsnun < 1}} }{\mathcal{N}(x)}} \mathcal{N}(x) \dx x  + \int_\Rd \underline{n}\epsnun \indicator_{\set{\underline{n}\epsnun < 1}} \abs*{x}^{2} \dx x \\
		  & \leq - \int_\Rd \frac{\underline{n}\epsnun \indicator_{\set{\underline{n}\epsnun < 1}}}{\mathcal{N}(x)} \log\prt*{\frac{\underline{n}\epsnun \indicator_{\set{\underline{n}\epsnun < 1}} }{\mathcal{N}(x)}} \mathcal{N}(x)  \dx x + C,
	\end{align*}
	having used the second-order moment bound from Proposition~\ref{lem:second_moment}. Applying Jensen's inequality to the first term, we observe
	\begin{align*}
		- \int_\Rd
		 & \frac{\underline{n}\epsnun \indicator_{\set{\underline{n}\epsnun < 1}}}{\mathcal{N}(x)} \log\prt*{\frac{\underline{n}\epsnun \indicator_{\set{\underline{n}\epsnun < 1}} }{\mathcal{N}(x)}} \mathcal{N}(x) \dx x                                                        \\
		 & \leq - \prt*{\int_\Rd \frac{\underline{n}\epsnun \indicator_{\set{\underline{n}\epsnun < 1}} }{\mathcal{N}(x)}(x) \mathcal{N}(x) \dx x} \log \prt*{\int_\Rd \frac{\underline{n}\epsnun \indicator_{\set{\underline{n}\epsnun < 1}} }{\mathcal{N}(x)} \mathcal{N} \dx x} \\
		 & \leq e^{-1},
	\end{align*}
	as $s\mapsto s \log(s) $ is convex. In conclusion, we obtain $J \leq C + e^{-1}$, whence
	\begin{equation*}
		\int_\Rd \underline{n}\epsnun \abs{\log \underline{n}\epsnun} \dx x \leq \norm{\underline{n}\epsnun}_{L^\infty(0,T; L^\infty(\Rd))} \norm{\underline{n}\epsnun}_{L^\infty(0,T;L^1(\Rd))}  + C,
	\end{equation*}
	which concludes the proof.
\end{proof}

\subsection{Compactness of solutions to the regularised equation}

Let $(\calK_h)_{0<h<1}$ be a family of nonnegative functions such that
\begin{align*}
	\supp{\calK_h} &\subset B_2(0), \quad \calK_h \in C^\infty(\R^d\setminus B_1(0)), \\
	\calK_h(x) & = \frac{1}{(|x|^2+h^2)^{d/2}}\quad \text{for $|x|\leq 1$}.
\end{align*}
The following lemma was proved in~\cite[Lemma~3.1]{BelgacemJabin}.
\begin{lemma}[Compactness criterion]
	\label{lem:CompactnessCriterion}
	Let $(u_k)$ be a bounded sequence in $L^p(\R^d\times(0,T))$ for some $1\leq p<\infty$. If $(\partial_t u_k)$ is uniformly bounded in $L^r(0,T;W^{-1,r})$ for some $r\geq 1$, and if
	\begin{equation*}
		\lim_{h\to0}\,\,\limsup_{k\to\infty}\,|\log h|^{-1}\,\int_0^T\int_{\R^{2d}} \calK_h(x-y)\abs*{u_k(x,t)-u_k(y,t)}^p\dx x\dx y\dx t = 0,
	\end{equation*}
	then $(u_k)$ is compact in $L^p_{\mathrm{loc}}(\R^d\times(0,T))$. Conversely, if $(u_k)$ is globally compact in $L^p$, then the above limit holds.
\end{lemma}

\begin{proposition}
	\label{prop:W_eps_Strong}
	The sequence $(W\epsnun)_{\epsilon>0}$ is compact in $L^1(0,T;L^1(\Rd))$. Consequently
	\begin{equation}
		\label{eq:Q_W_quantity}
		\lim_{h\to0}\,\limsup_{\epsilon\to0}\,|\log h|^{-1}\,\int_0^T\int_{\R^{2d}} \calK_h(x-y)\abs*{W\epsnun(x) - W\epsnun(y)}\dx x\dx y\dx t = 0.
	\end{equation}
\end{proposition}
\begin{proof}
	From the Brinkman equation we know that $W\epsnun$ is uniformly bounded in any $L^\infty(0,T;W^{1,q}(\Rd))$ for $q\in[1,\infty]$.
	Moreover, since $\partial_t W\epsnun = K_\nu \star \partial_t \underline{n}\epsnun$, we have $\partial_t W\epsnun \in L^2(0,T;H^{-1}(\R^d))$ uniformly, using Proposition~\ref{prop:TimeDerivative}. Hence, by the Aubin-Lions lemma, $W\epsnun$ is compact in $L^1(0,T;L^1_{\mathrm{loc}}(\Rd))$. To obtain global compactness we argue that the sequence is equi-tight. Indeed, testing the Brinkman equation with $\frac12|x|^2$ and integrating by parts, we see
	\begin{equation*}
		-\nu d \int_{\Rd} W\epsnun\dx x + \frac12\int_{\Rd} |x|^2 W\epsnun \dx x = \frac12 \int_{\Rd} |x|^2 \underline{n}\epsnun \dx x.
	\end{equation*}
	Using Proposition~\ref{lem:second_moment}, we deduce that $W\epsnun$ has a uniformly finite second moment, implying global compactness. Finally, Eq.~\eqref{eq:Q_W_quantity} holds by Lemma~\ref{lem:CompactnessCriterion}.
\end{proof}

\begin{proposition}
	\label{prop:n_epsnu_compact}
	The family $(\underline{n}\epsnun)_{\epsilon > 0}$ is compact in $L^1(0,T;L^{1}(\Rd))$.
\end{proposition}

\begin{proof}
	Let us denote
	\begin{equation*}
		\underline{Q}_h(t):=\int_{\R^{2d}} \calK_h(x-y)\abs*{\underline{n}\epsnun(x) - \underline{n}\epsnun(y)}\dx x\dx y,
	\end{equation*}
	and
	\begin{equation*}
		Q_h(t) := \int_{\R^{2d}} \calK_h(x-y)\int_0^1\abs*{n\epsnun(x) - n\epsnun(y)}\dx a\dx x\dx y.
	\end{equation*}
	Using Eq~\eqref{eq:regularized_Brinkman_a}, we derive
	\begin{align*}
		 \fpartial t&\abs*{n\epsnun(x) - n\epsnun(y)} \\[0.5em]
		 &\quad- \nabla_x\cdot\prt*{\nabla W\epsnun(x)\abs*{n\epsnun(x) - n\epsnun(y)}}                                            \\[0.5em]
		 &\quad - \nabla_y\cdot\prt*{\nabla W\epsnun(y)\abs*{n\epsnun(x) - n\epsnun(y)}}                                                                                       \\[0.5em]
		 &\quad + \frac12\prt*{\Delta W\epsnun(x)+\Delta W\epsnun(y)}\abs*{n\epsnun(x) - n\epsnun(y)}                                                                          \\[0.5em]
		 &\quad - \frac12\prt*{\Delta W\epsnun(x)-\Delta W\epsnun(y)}\prt*{n\epsnun(x) + n\epsnun(y)}\sigma \\[0.5em]
		 &\quad - \epsilon\prt*{\Delta_x+\Delta_y}\abs*{n\epsnun(x) - n\epsnun(y)} \\[0.5em]
		 & \leq \prt*{n\epsnun(x)G(\underline{n}\epsnun(x)) - n\epsnun(x)G(\underline{n}\epsnun(x))}\sigma,
	\end{align*}
	where $\sigma = \sigma(x,t;a):=\sign(n\epsnun(x,t;a)-n\epsnun(y,t;a))$.
	Multiplying by $\calK_h(x-y)$ and integrating, we obtain, using the symmetry of $\calK$,
	\begin{align*}
		 & \ddt \int_{\R^{2d}} \calK_h(x-y)\int_0^1\abs*{n\epsnun(x) - n\epsnun(y)}\dx a \dx x\dx y                                                          \\
		 & \leq - \int_{\R^{2d}} \calK_h(x-y)\Delta W\epsnun(x)\int_0^1\abs*{n\epsnun(x) - n\epsnun(y)}\dx a\dx x\dx y                                       \\
		 & \quad +\int_{\R^{2d}} \calK_h(x-y)\prt*{\Delta W\epsnun(x)-\Delta W\epsnun(y)} \int_0^1 n\epsnun(x)\sigma\dx a\dx x\dx y                          \\
		 & \quad -2\int_{\R^{2d}} \nabla\calK_h(x-y)\cdot\prt*{\nabla W\epsnun(x)-\nabla W\epsnun(y)}\int_0^1\abs*{n\epsnun(x) - n\epsnun(y)}\dx a\dx x\dx y \\
		 & \quad +2\epsilon \int_{\R^{2d}} \Delta \calK_h(x-y)\int_0^1\abs*{n\epsnun(x) - n\epsnun(y)}\dx a \dx x\dx y                                       \\
		 & \quad +\int_{\R^{2d}}\calK_h(x-y)\int_0^1 G_N(\underline{n}\epsnun(x))\abs*{n\epsnun(x) - n\epsnun(y)}\dx a \dx x\dx y                            \\
		 & \quad +\int_{\R^{2d}}\calK_h(x-y)\int_0^1 n\epsnun(y)\prt*{G_N(\underline{n}\epsnun(x))-G_N(\underline{n}\epsnun(y))}\sigma\dx a \dx x\dx y\\
		 & =: \mathcal I_1 + \ldots + \mathcal I_6.
	\end{align*}
	We now observe the following bounds:
	\begin{align*}
		 \mathcal I_2  &\leq \abs*{\int_{\R^{2d}} \calK_h(x-y)\prt*{\Delta W\epsnun(x)-\Delta W\epsnun(y)} \int_0^1 n\epsnun(x)\sigma\dx a\dx x\dx y}                 \\
		 & \leq \frac{1}{\nu}\norm*{\underline{n}\epsnun}_{L^\infty(0,T; L^\infty(\Rd))}\int_{\R^{2d}} \calK_h(x-y)\abs*{W\epsnun(x)-W\epsnun(y)}\dx x\dx y    \\
		 & \quad + \frac{1}{\nu}\int_{\R^{2d}} \calK_h(x-y)\abs*{\underline{n}\epsnun(x)-\underline{n}\epsnun(y)}\underline{n}\epsnun(x)\dx x\dx y,
	\end{align*}
	having used Brinkman's law, and
	\begin{align*}
		\mathcal I_4 &\leq 2\epsilon \int_{\R^{2d}} \Delta \calK_h(x-y)\int_0^1\abs*{n\epsnun(x) - n\epsnun(y)}\dx a \dx x\dx y\\
		&\leq C\norm*{\underline{n}\epsnun}_{L^\infty(0,T; L^\infty(\Rd))}\frac{\epsilon}{h^2}, 
	\end{align*}
	from directly estimating the Laplacian of the kernel $\calK_h$.
	For the terms involving the growth rates, we have
	\begin{align*}
		 \mathcal I_5 &= \int_{\R^{2d}}\calK_h(x-y)\int_0^1 G_N(\underline{n}\epsnun(x))\abs*{n\epsnun(x) - n\epsnun(y)}\dx a \dx x\dx y                                          \\
		 & \leq \sup_{a\in[0,1]}\sup_{n\in[0,\bar n]} |G(n;a)| \int_{\R^{2d}}\calK_h(x-y)\int_0^1 \abs*{n\epsnun(x) - n\epsnun(y)}\dx a \dx x\dx y,
	\end{align*}
	and
	\begin{align*}
		 \mathcal I_6 &= \int_{\R^{2d}}\calK_h(x-y)\int_0^1 n\epsnun(y)\prt*{G_N(\underline{n}\epsnun(x))-G_N(\underline{n}\epsnun(y))}\sigma\dx a \dx x\dx y                                                    \\
		 & \leq \sup_{a\in[0,1]}\sup_{n\in[0,\bar n]} |\partial_n G(n;a)| \int_{\R^{2d}}\calK_h(x-y)\underline{n}\epsnun(y)\abs*{\underline{n}\epsnun(x)-\underline{n}\epsnun(y)} \dx x\dx y \\
		 & \leq \alpha \norm*{\underline{n}\epsnun}_{L^\infty(0,T;L^\infty(\Rd))}\int_{\R^{2d}}\calK_h(x-y)\abs*{\underline{n}\epsnun(x)-\underline{n}\epsnun(y)} \dx x\dx y.
	\end{align*}
	Finally, we consider the main commutator
    \begin{equation*}
        \mathcal I_3 = \int_{\R^{2d}} \nabla\calK_h(x-y)\cdot\prt*{\nabla W\epsnun(x)-\nabla W\epsnun(y)}\int_0^1\abs*{n\epsnun(x) - n\epsnun(y)}\dx a\dx x\dx y.
    \end{equation*}
    The careful treatment of commutators of this form is one of the main contributions of the papers~\cite{BelgacemJabin, BelgacemJabin2}. Since our case is covered by these results, we omit the details, and refer the reader to the aforementioned papers for full proof. Applying~\cite[Proposition~13]{BelgacemJabin2}, we deduce
	\begin{align*}
		\mathcal I_3 & \leq C_1\norm{D^2 W\epsnun}_{L^\infty(0,T; L^2(\Rd))} \abs*{\log h}^{1/2}\\
		& \quad + C_2\norm*{\Delta W\epsnun}_{L^\infty(0,T;L^\infty(\R^d))}\int_{\R^{2d}}\calK_h(x-y)\int_0^1\abs*{n\epsnun(x) - n\epsnun(y)}\dx a \dx x\dx y,
	\end{align*}
    where the constant $C_1$ depends on the norms of $n\epsnun$.
	Putting all the estimates together, we have
	\begin{align*}
		 & \ddt \int_{\R^{2d}} \calK_h(x-y)\int_0^1\abs*{n\epsnun(x) - n\epsnun(y)}\dx a \dx x\dx y                                                 \\
		 & \leq C\int_{\R^{2d}}\calK_h(x-y)\int_0^1\abs*{n\epsnun(x) - n\epsnun(y)}\dx a \dx x\dx y\\
		 & \quad + C\int_{\R^{2d}}\calK_h(x-y)\abs*{\underline{n}\epsnun(x) - \underline{n}\epsnun(y)} \dx x\dx y\\
		 & \quad +C\int_{\R^{2d}}\calK_h(x-y)\abs*{W\epsnun(x) - W\epsnun(y)} \dx x\dx y\\
        &\quad + C\frac{\epsilon}{h^2} + C\abs*{\log h}^{1/2},
	\end{align*}
    where we stress that the constants may depend unfavourably on $\nu>0$ but are independent of $N$ and $\epsilon$.
	Proceeding in the exact same way with the equation for $\underline{n}\epsnun$, we obtain
	\begin{align*}
		 & \ddt \int_{\R^{2d}} \calK_h(x-y)\abs*{\underline{n}\epsnun(x) - \underline{n}\epsnun(y)}\dx x\dx y                                       \\
		 & \leq +C\int_{\R^{2d}}\calK_h(x-y)\int_0^1\abs*{n\epsnun(x) - n\epsnun(y)}\dx a \dx x\dx y                                               \\
		 & \quad + C\int_{\R^{2d}}\calK_h(x-y)\abs*{\underline{n}\epsnun(x) - \underline{n}\epsnun(y)} \dx x\dx y                                   \\
		 & \quad +C\int_{\R^{2d}}\calK_h(x-y)\abs*{W\epsnun(x) - W\epsnun(y)} \dx x\dx y\\
        &\quad +C\frac{\epsilon}{h^2} + C\abs*{\log h}^{1/2}.
	\end{align*}
	Hence, the sum of the two compactness quantities satisfies
	\begin{equation}\label{eq:Compactness_Gronwall}
    \begin{aligned}
		\ddt \brk*{Q_h(t) + \underline{Q}_h(t)} & \leq C\frac{\epsilon}{h^2} + C\abs*{\log h}^{1/2} + C\brk*{Q_h(t) + \underline{Q}_h(t)} \\
		& \quad +C\int_{\R^{2d}}\calK_h(x-y)\abs*{W\epsnun(x) - W\epsnun(y)} \dx x\dx y.
	\end{aligned}
    \end{equation}
	Applying Gronwall's lemma, we have
	\begin{align*}
		Q_h(t) + \underline{Q}_h(t) & \leq C\frac{\epsilon}{h^2} + C\abs*{\log h}^{1/2} + C\brk*{Q_h(0) + \underline{Q}_h(0)}     \\
		& \quad +C\int_0^T\int_{\R^{2d}}\calK_h(x-y)\abs*{W\epsnun(x) - W\epsnun(y)} \dx x\dx y\dx t.
	\end{align*}
	Multiplying by $|\log h|$, taking the limit superior over $\epsilon>0$, and letting $h\to0$, we deduce that
	\begin{equation*}
		\lim_{h\to0}\,\limsup_{\epsilon\to0}\,|\log h|^{-1}\,\sup_{t\in[0,T]}\,\underline{Q}_h(t) = 0,
	\end{equation*}
	where we used the fact that the initial data is compact in $\epsilon$ by construction, and~\eqref{eq:Q_W_quantity} from Proposition~\ref{prop:W_eps_Strong}.
	Using Lemma~\ref{lem:CompactnessCriterion}, we conclude that the sequence $(\underline{n}\epsnun)_{\epsilon>0}$ is compact in $L^1(0,T;L^1_{\mathrm{loc}}(\R^d))$. 
    The same limit is true for $Q_h$. Observing that
	\begin{align*}
		 Q_h(t) & = \int_{\R^{2d}} \calK_h(x-y)\int_0^1\abs*{n\epsnun(x) - n\epsnun(y)}\dx a \dx x\dx y                                                             \\
		 & =\int_0^1\int_{\R^{2d}} \calK_h(x-y)\abs*{n\epsnun(x) - n\epsnun(y)}\dx x\dx y\dx a                                                             \\
		 & = \int_0^1\int_{\R^{2d}} \calK_h(x-y)\sum_{i=1}^N\abs*{n\i\epsnun(x) - n\i\epsnun(y)}\indicator_{(\frac{i-1}{N},\frac{i}{N}]}(a)\dx x\dx y\dx a \\
		 & = \frac{1}{N}\sum_{i=1}^N\int_{\R^{2d}} \calK_h(x-y)\abs*{n\i\epsnun(x) - n\i\epsnun(y)}\dx x\dx y,
	\end{align*}
	we have, for each $i=1,\dots,N$,
	\begin{equation*}
		\lim_{h\to0}\,\limsup_{\epsilon\to0}\,|\log h|^{-1}\,\sup_{t\in[0,T]}\,\int_{\R^{2d}} \calK_h(x-y)\abs*{n\i\epsnun(x) - n\i\epsnun(y)}\dx x\dx y = 0,
	\end{equation*}
	which implies local compactness of each of the densities $n\i\epsnun$ as $\epsilon\to0$.  Global compactness is obtained by virtue of the second-moment bound from Proposition~\ref{lem:second_moment}.
\end{proof}

As a consequence of the above proposition in conjunction with the uniform $L^\infty$ bounds, we deduce that the densities $n\i\epsnun$ are compact in $L^{q}(0,T,L^{p}(\Rd))$ for all $p,q \in [1,\infty)$. Moreover, since
\begin{equation*}
	\norm*{\grad W\epsnun -\grad W\nun}_{L^q(0,T;L^p(\Rd))} \leq \norm{\grad K_{\nu}}_{L^{1}(\Rd)}
	\norm{\underline{n}\epsnun- \underline{n}\nun}_{L^q(0,T;L^p(\Rd))},
\end{equation*}
the same range of strong convergence is true for $\nabla W\epsnun$.

Putting together the results obtained above, we can formulate the following summary:
\begin{corollary}[Convergence as $\epsilon \to 0$]
\label{cor:EpsConvergence}
	Upon passing to a subsequence, we can find $ n\i\nun \in L^{\infty}(0,T;L^{1}\cap L^{\infty}(\Rd)) $ such that
	\begin{itemize}
		\item $n\i\epsnun \stackrel{\star}{\rightharpoonup} n\i\nun$ in  $L^{\infty}(0,T;L^{p}(\Rd))$,  for all  $1\leq p \leq \infty$, 
        \item $n\i\epsnun\to n\i\nun$ in $L^q(0,T,L^p(\Rd))$, for all $p,q \in [1,\infty)$,
	\end{itemize}
	and, by their definition,
	\begin{itemize}
		\item $n\epsnun(\cdot, \cdot\,; a) \stackrel {\star}{\rightharpoonup} n_{\nun}(\cdot, \cdot\,; a)$ in  $L^{\infty}(0,T;L^{p}(\Rd))$,  for all  $1\leq p \leq \infty$ and a.e. $a\in [0,1]$, 
        \item $n\epsnun\to n\nun$ in $L^q(0,T,L^p(\Rd))$, for all $p,q \in [1,\infty)$,  
  		\item $\underline n\epsnun \stackrel {\star}{\rightharpoonup} \underline n_{\nun}$ in  $L^{\infty}(0,T;L^{p}(\Rd))$,  for all  $1\leq p \leq \infty$,
        \item $\underline n\epsnun\to \underline n\nun$ in $L^q(0,T,L^p(\Rd))$, for all $p,q \in [1,\infty)$,
        \item $\underline n\epsnun\to \underline n\nun$ almost everywhere in $\Rd\times[0,T]$.
		\item $\partial_{t} \underline n\epsnun \rightharpoonup \partial_{t} \underline n_{\nun}$, in $L^{2}(0,T;H^{-1}(\Rd))$,
	\end{itemize} 
	as well as 
	\begin{itemize}
        \item $W\epsnun \to W\nun$ in $L^q(0,T,W^{1,p}(\Rd))$, for all $p,q \in [1,\infty)$.
	\end{itemize} 
\end{corollary}
With these convergence results it is now trivial to pass to the limit in the weak form of System~\eqref{eq:regularised_Brinkman_i} and deduce that the tuple $(n\i\nun, W\nun)_{i=1}^N$ is a weak solution of System~\eqref{eq:Brinkman_i}. Thus, the proof of Lemma~\ref{lem:Existence} is complete. 

We conclude this subsection with the following observation regarding continuity in time of the solutions.

\begin{proposition}
    \label{prop:TimeContinuity}
    For each $\nu>0$, the function $\underline{n}_{\nu, N}$ constructed in the previous section belongs to $C([0,T]; L^2(\Rd))$, after possibly changing it on a set of measure zero. Moreover, we have $\underline{n}_{\epsilon, \nu, N}(t) \to \underline{n}_{\nu, N}(t)$, for every $t\in[0,T]$, and $\max_{t\in[0,T]} \norm{\underline{n}_{\epsilon, \nu, N}(t)}_{L^2(\Rd)} \leq C$, where $C>0$ is independent of $\epsilon$, $\nu$, and $N$.
\end{proposition}
\begin{proof}
The argument is based on a generalised Arzel\`a-Ascoli theorem \cite[Theorem 47.1]{Mun2000}. First, let us observe that we can establish the uniform-in-time $L^2$-control, by considering
\begin{align*}
    \frac12 \ddt \norm{\underline{n}_{\epsilon, \nu, N}}_{L^2(\Rd)}^2 \leq C \norm{\underline{n}_{\epsilon, \nu, N}}_{L^2(\Rd)}^2,
\end{align*}
where the constant $C>0$ is independent of $\nu$, $\epsilon$, and $N$. An application of Gronwall's lemma yields 
$$
    \sup_{0 \leq t \leq T} \norm{\underline{n}_{\epsilon, \nu, N}}_{L^2(\Rd)}^2 \leq C.
$$
In addition, with the fact that
\begin{align*}
    \sup_{t\in[0,T]} & |\log h|^{-1}\iint_{\R^{2d}}  \calK_h(x-y) |\underline{n}_{\epsilon, \nu, N}(x,t) - \underline{n}_{\epsilon, \nu, N}(y,t)|^2\dx x \dx y \\
    &\leq C \sup_{t\in[0,T]} |\log h|^{-1}\iint_{\R^{2d}} \calK_h(x-y) |\underline{n}_{\epsilon, \nu, N}(x,t) - \underline{n}_{\epsilon, \nu, N}(y,t)|\dx x \dx y \to 0,
\end{align*}
by Proposition \ref{prop:n_epsnu_compact}, we can conclude that the set $(\underline{n}_{\epsilon, \nu, N}(t))_{\epsilon>0}$ is relatively compact in $L^2(\Rd)$ for each $t\in[0,T]$. It remains to show equi-continuity in $C([0,T];L^2(\Rd))$. To this end, we consider
\begin{align*}
	\norm{\underline{n}_{\epsilon, \nu, N}(t+h) - \underline{n}_{\epsilon, \nu, N}(t)}_{L^2(\Rd)}^2 
	&\lesssim  \norm{\underline{n}_{\epsilon, \nu, N, \alpha}(t+h) - \underline{n}_{\epsilon, \nu, N}(t+h)}_{L^2(\Rd)}^2\\
	&\qquad  + \norm{\underline{n}_{\epsilon, \nu, N, \alpha}(t+h) - \underline{n}_{\epsilon, \nu, N, \alpha}(t)}_{L^2(\Rd)}^2\\
	&\qquad + \norm{\underline{n}_{\epsilon, \nu, N, \alpha}(t) - \underline{n}_{\epsilon, \nu, N}(t)}_{L^2(\Rd)}^2,
\end{align*}
where we used the triangular inequality with the mollified densities
$$
    \underline{n}_{\epsilon, \nu, N, \alpha}(t) := \bar \calK_\alpha \star_x \underline{n}_{\epsilon, \nu, N}(t),
$$
for every $t\in[0,T]$. Here, the kernel $\bar \calK_\alpha$ is the one introduced in Lemma \ref{lem:CompactnessCriterion}, normalised to have unit mass. Now, we observe that
\begin{align*}
	\norm{\underline{n}_{\epsilon, \nu, N, \alpha}(t) - \underline{n}_{\epsilon, \nu, N}(t)}_{L^2(\Rd)}^2 
	&\leq \int_\Rd \abs*{\int_\Rd \bar \calK_{\alpha}(x-y)(\underline{n}_{\epsilon, \nu, N}(y, t) - \underline{n}_{\epsilon, \nu, N}(x, t) )\dx y}^2 \dx x\\
	&\leq \iint_{\R^{2d}} \bar \calK_{\alpha}(x-y)\abs*{\underline{n}_{\epsilon, \nu, N}(y, t) - \underline{n}_{\epsilon, \nu, N}(x, t)}^2 \dx y \dx x\\
	&\leq C \norm{\calK_{\alpha}}_{L^1(\Rd)} \underline{Q}_{\alpha}(t),
\end{align*}
having used Jensen's inequality, the uniform $L^\infty$-bounds, and the notation from Proposition \ref{prop:n_epsnu_compact}. Similarly, we obtain 
\begin{align*}
	\norm{\underline{n}_{\epsilon, \nu, N, \alpha}(t+h) - \underline{n}_{\epsilon, \nu, N}(t+h)}_{L^2(\Rd)}^2 
	&\leq C \norm{\calK_{\alpha}}_{L^1(\Rd)}\underline{Q}_{\alpha}(t+h).
\end{align*}
Finally, let us address the second term. We find
\begin{align*}
	&\norm{\underline{n}_{\epsilon, \nu, N, \alpha}(t+h) - \underline{n}_{\epsilon, \nu, N, \alpha}(t)}_{L^2(\Rd)}^2 \\
	&\quad = \int_\Rd \brk{\underline{n}_{\epsilon, \nu, N, \alpha}(t+h) - \underline{n}_{\epsilon, \nu, N, \alpha}(t)}\int_t^{t+h} \partial_s \underline{n}_{\epsilon, \nu, N, \alpha}(s)\dx s\dx x\\
	&\quad = \int_t^{t+h} \int_\Rd  \brk{\underline{n}_{\epsilon, \nu, N, \alpha}(t+h) - \underline{n}_{\epsilon, \nu, N, \alpha}(t)}  \partial_s \underline{n}_{\epsilon, \nu, N, \alpha}(s)\dx x\dx s\\
	&\quad \leq C \norm{\underline{n}_{\epsilon, \nu, N, \alpha}}_{L^\infty(0,T;H^1(\Rd))} \norm{\partial_t\underline{n}_{\epsilon, \nu, N, \alpha}}_{L^2(t,t+h;H^{-1}(\Rd))}\\
    &\quad \leq C \norm{\underline{n}_{\epsilon, \nu, N}}_{L^\infty(0,T;L^2(\Rd))} \norm{\partial_t\underline{n}_{\epsilon, \nu, N, \alpha}}_{L^2(0,T;H^{-1}(\Rd))}\sqrt{h}/\alpha\\
	&\quad \leq C \sqrt{h}/\alpha.
\end{align*}
Thus, we conclude that 
\begin{align*}
	\norm{\underline{n}_{\epsilon, \nu, N}(t+h) - \underline{n}_{\epsilon, \nu, N}(t)}_{L^2(\Rd)}^2 \leq C\sup_{\epsilon>0} \sup_{t\in[0,T]}|\log \alpha|^{-1} \underline Q_\alpha(t) + C\sqrt{h}/\alpha.
\end{align*}
Choosing $\alpha=h^{1/4}$, we have that
\begin{align*}
	\norm{\underline{n}_{\epsilon, \nu, N}(t+h) - \underline{n}_{\epsilon, \nu, N}(t)}_{L^2(\Rd)}^2 \to 0,
\end{align*}
as $h\to 0$ uniformly in $\epsilon>0$. Now, by the Arzel\`a-Ascoli theorem, there exists a function $g \in C([0,T];L^2(\Rd))$ and a subsequence up to which
$$
	\underline{n}_{\epsilon, \nu, N}(t) \to g(t),
$$
for all $t\in[0,T]$. Thus $g$ is the time-continuous representative of $\underline{n}_{\nu, N}$ and, henceforth, we identify $\underline{n}_{\nu, N}$ with $g$. In particular, we have $\underline{n}_{\epsilon, \nu, N}(T) \to \underline{n}_{\nu, N}(T)$.
\end{proof}

\begin{remark}
    \label{rmk:ContinuityL^1}
    For the subsequent analysis, let us point out that the continuity of $\underline{n}_{\nu, N}$ in $C([0,T];L^2(\Rd))$  and the uniform boundedness in $L^\infty(0,T;L^1(\Rd, (1 + |x|^2)\dx x)$ implies that $\underline{n}_{\nu, N} \in C([0,T];L^1(\Rd))$ and $\norm{\underline{n}_{\nu, N}(T)}_{L^1(\Rd, (1+|x|^2)\dx x)} \leq C$, uniformly in $\nu$ and $N$.
\end{remark}

\subsection{The entropy inequality}
We are now in a position to prove the entropy inequality.
\begin{proposition}[Entropy Inequality]\label{lem:entropy_inequality}
	The limit $\underline{n}\nun$ of the sequence $(\underline{n}\epsnun)_{\epsilon>0}$ constructed in the previous section satisfies the entropy inequality
	\begin{equation}
		\label{eq:entropy_inequality}
		\begin{split}
			\mathcal{H}[\underline{n}\nun](T) - \mathcal{H}[\underline{n}_{N}^{\mathrm{in}}]
			 & - \int_0^T \!\!\! \int_\Rd  \underline{n}\nun \Lap W\nun \dx x \dx t \\
			 & \leq\int_0^T \!\!\! \int_\Rd \log \underline{n}\nun
			\int_{0}^{1} n\nun G_{N}(\underline{n}\nun;a)\dx{a} \dx x \dx t.
		\end{split}
	\end{equation}
\end{proposition}
\begin{proof}
	Let $\delta > 0$ and $\phi \in C^{\infty}_{c} (\R^{d})$,
	$\phi \geq 0 $. We consider the following regularised form of the entropy functional
	\begin{equation*}
		\begin{aligned}
			\mathcal{H}_{\delta}^{\phi}[\underline{n}\epsnun](t) \coloneqq \int_{\R^d}^{}
			\left( \underline{n}\epsnun(x,t) + \delta \right)(\log(\underline{n}\epsnun(x,t)+\delta)-1)
			\phi(x)\dx x,
		\end{aligned}
	\end{equation*}
	and by $\mathcal{H}^{\phi}[f]$ we denote the above functional with $\delta=0$.
	Given the regularity of $ \underline{n} \epsnun $, the weak form of Eq.~\eqref{eq:regularised-equation} can be formulated as
	\begin{equation*}
		\begin{aligned}
			\int_0^T \!\!\! \int_\Rd  \partialt{\underline{n}\epsnun}\varphi(x,t)
			 & + \underline{n}\epsnun \grad \varphi(x,t) \cdot \grad W\epsnun \dx{x}\dx{t} + \epsilon \int_0^T \!\!\! \int_\Rd \grad \underline{n}\epsnun \cdot \grad \varphi(x,t)\dx{x}\dx{t} \\
			 & = \int_0^T \!\!\! \int_\Rd \varphi(x,t) \int_{0}^{1}  n\epsnun G_{N}(\underline{n}\epsnun;a) \dx{a}\dx{x}\dx{t},
		\end{aligned}
	\end{equation*}
	for any $\varphi \in L^2(0,T;H^1(\R^d))$. Choosing
	\begin{equation*}
		\begin{aligned}
			\varphi(x,t) \coloneqq \log(\underline{n}\epsnun + \delta) \phi(x),
		\end{aligned}
	\end{equation*}
	we obtain
	\begin{equation}\label{eq:entropy_estimate_split}
		\begin{aligned}
			\int_0^T \dfrac{\mathup{d}{}}{\mathup{dt}} \mathcal{H}^{\phi}_{\delta} [\underline{n}\epsnun](t)
			\dx t+ I^{\epsilon, \delta}_{1} + I^{\epsilon, \delta}_{2} +
			I^{\epsilon, \delta}_{3} + I^{\epsilon, \delta}_{4} = I^{\epsilon,
					\delta}_{5},
		\end{aligned}
	\end{equation}
	where
	\begin{equation*}
		\begin{aligned}
			I^{\epsilon, \delta}_{1} & =\int_0^T \!\!\!\int_{\R^{d}}^{} \frac{\underline{n}\epsnun}{\underline{n}\epsnun+
			\delta} \grad \underline{n}\epsnun \cdot \grad  W\epsnun \phi \dx x\dx t,                                                                                           \\
			I^{\epsilon, \delta}_{2} & = \int_0^T \!\!\! \int_\Rd \underline{n}\epsnun \log(\underline{n}\epsnun+\delta)\grad
			\phi \cdot \grad W\epsnun \dx x\dx t,                                                                                                                               \\
			I^{\epsilon, \delta}_{3} & = \epsilon \int_0^T \!\!\! \int_\Rd \frac{\abs{\grad
			\underline{n}\epsnun}^{2}}{\underline{n}\epsnun+\delta} \phi \dx x \dx t,                                                                                           \\
			I^{\epsilon, \delta}_{4} & = \epsilon \int_0^T \!\!\! \int_\Rd \log(\underline{n}\epsnun
			+\delta)\grad \phi \cdot \grad \underline{n}\epsnun \dx x \dx t,                                                                                                    \\
			I^{\epsilon, \delta}_{5} & = \int_0^T \!\!\! \int_\Rd \log(\underline{n}\epsnun+\delta) \phi \int_{0}^{1}  n\epsnun G_N(\underline{n}\epsnun;a) \dx{a}\dx{x}\dx{t}.
		\end{aligned}
	\end{equation*}

	Now we investigate each term individually. Starting with $ I_{1}^{\epsilon,\delta}$ we see
	\begin{equation*}
		\begin{aligned}
			I_{1}^{\epsilon,\delta}
			= & \int_0^T \!\!\!\int_{\R^{d}}^{} \frac{\underline{n}\epsnun}{\underline{n}\epsnun+
			\delta} \grad \underline{n}\epsnun \cdot \grad  W\epsnun \phi \dx x\dx t,                         \\
			= & \int_0^T \!\!\! \int_{\R^{d}}^{} \grad \underline{n}\epsnun \cdot \grad W\epsnun \phi
			\dx x\dx t-  \delta \int_0^T \!\!\!\int_{\R^{d}}^{} \frac{\grad \underline{n}\epsnun \cdot
			\grad  W\epsnun}{\underline{n}\epsnun+ \delta}  \phi \dx x\dx t,                                  \\
			= & -\int_0^T \!\!\! \int_\Rd  \underline{n}\epsnun \Lap W\epsnun \phi
			\dx x\dx t- \int_0^T\!\!\!\int_\Rd \underline{n}\epsnun\nabla W\epsnun\cdot\nabla\phi \dx x \dx t \\
			  & \; + \delta \int_0^T \!\!\!\int_{\R^{d}}^{}  \log(\underline{n}\epsnun
			+ \delta) \Lap  W\epsnun \phi \dx x\dx t +\delta \int_0^T \!\!\! \int_\Rd
			\log(\underline{n}\epsnun + \delta) \grad
			W\epsnun\cdot \grad \phi \dx x \dx t.
		\end{aligned}
	\end{equation*}
	Passing to the limit $\epsilon\to 0$, we readily obtain
	\begin{equation*}
		\begin{aligned}
			I_1^{\epsilon,\delta}\to I_{1}^{\delta}
			= & -\int_0^T \!\!\! \int_\Rd  \underline{n}\nun \Lap W\nun \phi
			\dx x\dx t- \int_0^T\!\!\!\int_\Rd \underline{n}\nun\nabla W\nun\cdot\nabla\phi \dx x \dx t                                                    \\
			  & \; + \delta \int_0^T \!\!\!\int_{\R^{d}}^{}  \log(\underline{n}\nun + \delta) \Lap  W\nun \phi \dx x\dx t+ \delta \int_0^T \!\!\! \int_\Rd
			\log(\underline{n}\nun + \delta) \grad W\nun\cdot \grad \phi \dx x \dx t.
		\end{aligned}
	\end{equation*}
	For the last two integrals, we observe that
	\begin{equation*}
		\begin{aligned}
			\delta \int_0^T \!\!\!\int_{\R^{d}}|\log(\underline{n}\epsnun + \delta)|
			  & |\Lap  W\epsnun| \phi \dx x\dx t                                                 \\
			+ & \delta \int_0^T \!\!\! \int_\Rd \abs{\log(\underline{n}\epsnun + \delta)} |\grad
			W\epsnun\cdot \grad \phi| \dx x \dx t\leq C \delta\abs{\log\delta}.
		\end{aligned}
	\end{equation*}
	Thus, when $\epsilon\to 0$ and $\delta\to 0$,
	\begin{equation*}
		I_1^{\epsilon,\delta}\to -\int_0^T \!\!\! \int_\Rd  \underline{n}\nun \Lap W\nun \phi
		\dx x\dx t- \int_0^T\!\!\!\int_\Rd \underline{n}\nun\nabla W\nun\cdot\nabla\phi \dx x \dx t.
	\end{equation*}
	Next, the convergence
	\begin{equation*}
		I_{2}^{\epsilon, \delta} \to \int_0^T \!\!\! \int_\Rd  \underline{n}\nun\log( \underline{n}\nun) \grad \phi \cdot \grad  W\nun  \dx x \dx t
	\end{equation*}
	follows easily from the dominated convergence theorem.
	The next term from Eq.~\eqref{eq:entropy_estimate_split}, $I_{3}^{\epsilon,\delta}$, is nonnegative and can be dropped in the limit.
	The fourth term from Eq.~\eqref{eq:entropy_estimate_split} is estimated by
	\begin{equation*}
		\begin{aligned}
			\abs*{I_{4}^{\epsilon,\delta}}
			 & \leq \epsilon \int_0^T \!\!\! \int_\Rd \abs{\grad \phi} \abs{\grad \underline{n}\epsnun}  \prt*{\abs{\log{\delta}} + \underline{n}\epsnun} \dx x \dx t                           \\[0.4em]
			 & \leq \epsilon\prt*{\norm{\underline{n}\epsnun}_{L^\infty(0,T;L^\infty(\Rd))} + |\log\delta|} \norm{\nabla\phi}_{L^2(\Rd)} \norm{\nabla \underline{n}\epsnun}_{L^2(0,T;L^2(\Rd))} \\[0.4em]
			 & \leq C\sqrt\epsilon \abs{\log(\delta)},
		\end{aligned}
	\end{equation*}
	having used Lemma~\ref{prop:regular_n_bounds}. Thus,
	as $ \epsilon \to 0$, and then $ \delta \to 0 $ we have $ I_{4}^{\epsilon,\delta} \to 0 $. The remaining term of Eq.~\eqref{eq:entropy_estimate_split} is given by
	\begin{equation*}
		\begin{aligned}
			I^{\epsilon, \delta}_{5} & = \int_0^T \!\!\! \int_\Rd \log(\underline{n}\epsnun+\delta)\phi(x) \int_{0}^{1}n\epsnun   G_N(\underline{n}\epsnun;a)\dx{a}\dx{x}\dx{t}.
		\end{aligned}
	\end{equation*}
	Notice that $n\epsnun G_N(\underline{n}\epsnun;a)$ converges strongly in $L^2$ to $n\nun G_N(\underline{n}\nun;a)$ for every $a\in[0,1]$ and $\abs*{n\epsnun G_N(\underline{n}\epsnun;a)} \leq \norm{G}_{C([0,\bar n]\times [0,1])}\bar n$.  Therefore
	\begin{equation*}
		\int_{0}^{1}n\epsnun   G_N(\underline{n}\epsnun;a)\dx{a} \to \int_{0}^{1}n\nun   G_N(\underline{n}\nun;a)\dx{a}
	\end{equation*}
	in $L^2(\Rd\times(0,T))$. It follows that, as $\epsilon\to 0$
	\begin{equation*}
		I^{\epsilon, \delta}_{5} \to \int_0^T \!\!\! \int_\Rd \log(\underline{n}\nun+\delta)\phi(x) \int_{0}^{1}n\nun   G_N(\underline{n}\nun;a)\dx{a}\dx{x}\dx{t}.
	\end{equation*}
	Then, as $\delta\to 0$, we deduce
	\begin{equation*}
		I^{\epsilon, \delta}_{5} \to \int_0^T \!\!\! \int_\Rd \log(\underline{n}\nun)\phi(x) \int_{0}^{1}n\nun   G_N(\underline{n}\nun;a)\dx{a}\dx{x}\dx{t},
	\end{equation*}
	using the dominated convergence theorem.
	Finally, we consider the term involving the time derivative.
	Since $\ds \underline{n}\epsnun \in C([0,T];L^{2}(\Rd))$, the mapping $\ds  t \mapsto \mathcal{H}^{\phi}_{\delta} [\underline{n}\epsnun(t)] $ is continuous in $[0,T]$. We therefore have
	\begin{align*}
		\int_0^T \dfrac{\mathup{d}{}}{\mathup{dt}} \mathcal{H}^{\phi}_{\delta} [\underline{n}\epsnun]
		\dx t= \mathcal{H}^{\phi}_{\delta} [\underline{n}\epsnun](T) - \lim_{s\to 0}\mathcal{H}^{\phi}_{\delta} [\underline{n}\epsnun](s) = \mathcal{H}^{\phi}_{\delta} [\underline{n}\epsnun](T) - \mathcal{H}^{\phi}_{\delta} [\underline{n}_{\epsilon, N}^{\mathrm{in}}].
	\end{align*}
	Then, since $\underline{n}\epsnun(T)\to\underline{n}\nun(T)$ in $L^2(\Rd)$, we have as $\epsilon\to 0$,
	\begin{align*}
		\mathcal{H}^{\phi}_{\delta} [\underline{n}\epsnun](T)-\mathcal{H}^{\phi}_{\delta}[\underline{n}_{\epsilon,N}^{\mathrm{in}}] \to \mathcal{H}^{\phi}_{\delta} [\underline{n}\nun](T)-\mathcal{H}^{\phi}_{\delta}[\underline{n}_{N}^{\mathrm{in}}].
	\end{align*}
 	Now, by the dominated convergence theorem, we obtain
	\begin{align*}
		\mathcal{H}^{\phi}_{\delta} [\underline{n}\nun](T) & =
		\int_\Rd \prt{\underline{n}\nun(T)+\delta}\prt{\log(\underline{n}\nun(T)+\delta)-1}\phi \dx x                                      \\
		                                                   & \to \int_\Rd \underline{n}\nun(T)\prt{\log \underline{n}\nun(T)-1}\phi \dx x,
	\end{align*}
	as $\delta\to0$.

	\medskip We have thus shown that, in the limit $\epsilon\to 0$ and, subsequently $\delta\to 0$, Eq.~\eqref{eq:entropy_estimate_split} becomes
	\begin{equation}
		\label{eq:Localised_Entropy_Inequality}
		\begin{aligned}
			\mathcal{H}^{\phi}[\underline{n}\nun](T) - \mathcal{H}^{\phi}[\underline{n}_{N}^{\mathrm{in}}]
			 & - \int_0^T \!\!\! \int_\Rd  \underline{n}\nun \Lap W\nun \phi \dx x \dx t-
			\int_0^T \!\!\! \int_\Rd  \underline{n}\nun \grad W\nun \cdot \grad \phi \dx x
			\dx t                                                                                                              \\
			 & \;\; + \int_0^T \!\!\! \int_\Rd  \underline{n}\nun\log \underline{n}\nun\grad \phi \cdot \grad W\nun\dx x \dx t \\
			 & \leq \int_0^T \!\!\! \int_\Rd \phi \log{\underline{n}\nun}
			\int_0^1 n\nun G_{N}(\underline{n}\nun;a)\dx a \dx x \dx t,
		\end{aligned}
	\end{equation}
	for any $ \phi \in C_{c}^{\infty}(\Rd),\, \phi \geq 0 $.
	Let us now choose $\phi=\chi_R$, where $\chi_R$ is a sequence of smooth cut-off functions such that $|\nabla \chi_R|\lesssim R^{-1}$. Then, using the $L^\infty L^1$-control of $\underline{n}\nun\log \underline{n}\nun$ from Proposition~\ref{lem:entropy-bounds}, we can pass to the limit $R\to\infty$ by the monotone convergence theorem to obtain Eq.~\eqref{eq:entropy_inequality}. \qedhere
\end{proof}

\section{The joint limit}
\label{sec:JointLimit}
Before we start discussing the joint limit let us prove a short convergence result of the initial data and the growth rates.
\begin{proposition}
	\label{prop:Data_Strong_N}
    The initial data of the $N$-system converges in the following sense:
    \begin{enumerate}
        \item For each $a\in[0,1]$, $n^\mathrm{in}_N(\cdot\,;a) \to n^\mathrm{in}(\cdot\,;a)$ in $L^p(\R^d)$, $1\leq p < \infty$,
        \item $\underline{n}^{\mathrm{in}}_N \to \underline{n}^{\mathrm{in}}:=\int_0^1 n^\mathrm{in}\dx a$ in $L^p(\Rd)$, $1\leq p < \infty$,
        \item up to a subsequence, $\mathcal{H}[\underline{n}_N^{\mathrm{in}}]\to \mathcal{H}[\underline{n}^{\mathrm{in}}]$.
    \end{enumerate}
    Likewise, the growth rate $G_N$ converges in the following sense:
    \begin{enumerate}
	    \setcounter{enumi}{3}
        \item For each $a\in [0,1]$, $G_N(\cdot\,;a)\to G(\cdot\,;a)$ uniformly on $I$ for each compact set $I\subset\R$,
        \item let $(m_N)_N$ be a sequence of functions such that $0 \leq m_N\leq \bar n$ and such that $m_N$ converges to $m$ in $L^p(\Rd\times(0,T))$. Then $G_N(m_N;a)$ converges to $G(m;a)$ in $L^q(\Rd\times(0,T))$ for any $p\leq q<\infty$, for any $a\in[0,1]$.
    \end{enumerate}
\end{proposition}
\begin{proof}
	\underline{Ad (1).}
	Let $a\in[0,1]$ be fixed and note that there is an integer $i(N)$ such that 
	$$
		a\in\left(\frac{i(N)-1}{N},\frac{i(N)}{N}\right],
	$$
	for each each $N$. Now, let $a_N:=\frac{i(N)}{N}$ and observe that then  $|a_N-a|<\frac{1}{N}$, so that $a_N\to a$.
	Next, consider
	\begin{align*}
		\norm*{n_N^{\mathrm{in}}(\cdot\,;a)-n^{\mathrm{in}}(\cdot\,;a)}_{L^p(\Rd)} 
		&= \norm*{n_N^{(i(N)),\mathrm{in}}(\cdot)-n^{\mathrm{in}}(\cdot\,;a)}_{L^p(\Rd)} \\
		&= \norm*{n^{\mathrm{in}}(\cdot\,;a_N)-n^{\mathrm{in}}(\cdot\,;a)}_{L^p(\Rd)}.
	\end{align*}
    Let $\delta>0$ and choose a compact set $K_\delta\subset\Rd$ such that
    \begin{align*}
        \norm*{n^{\mathrm{in}}(\cdot\,;a_N)-n^{\mathrm{in}}(\cdot\,;a)}_{L^p(\Rd)} &\leq \norm*{n^{\mathrm{in}}(\cdot\,;a_N)-n^{\mathrm{in}}(\cdot\,;a)}_{L^p(K_\delta)} + \norm*{n^{\mathrm{in}}(\cdot\,;a_N)-n^{\mathrm{in}}(\cdot\,;a)}_{L^p(\Rd\setminus K_\delta)}\\
        &\leq \norm*{n^{\mathrm{in}}(\cdot\,;a_N)-n^{\mathrm{in}}(\cdot\,;a)}_{L^p(K_\delta)} + \frac{\delta}{3},
    \end{align*}
    where we used the uniform moment bound for $n^{\mathrm{in}}$ and the $L^\infty$-bound $n^{\mathrm{in}}\leq \bar n$.
    Regarding the first term, we use the continuity of $n^{\mathrm{in}}$ in the $a$-variable, Property~\eqref{eq:DataLpBound}, and the dominated convergence theorem.

	\bigskip \underline{Ad (2).} 
	For $\underline{n}^{\mathrm{in}}_N = \int_0^1 n^{\mathrm{in}}_N\dx a$, we use Minkowski's inequality to estimate
	\begin{align*}
		\norm*{\underline{n}^{\mathrm{in}}_N-\underline{n}^{\mathrm{in}}}_{L^p(\Rd)}
		\leq \int_0^1  \norm*{n^{\mathrm{in}}_N(\cdot\,;a)-n^{\mathrm{in}}(\cdot\,;a)}_{L^p(\Rd)} \dx a.
	\end{align*}
	The integrand converges for each $a$ and is bounded by
	\begin{align*}
		\norm{n^{\mathrm{in}}_N(\cdot\,;a)-n^{\mathrm{in}}(\cdot\,;a)}_{L^p(\Rd)} \leq 2\sup_{a\in[0,1]}\norm{n^{\mathrm{in}}(\cdot\,; a)}_{L^p(\Rd)} \leq C.
	\end{align*}
	Therefore, by the dominated convergence theorem, we deduce $\underline{n}^{\mathrm{in}}_N \to \underline{n}^{\mathrm{in}}$, in $L^p(\Rd)$.
	
    \bigskip \underline{Ad (3).} Recall that
    \begin{equation*}
        \mathcal{H}[\underline{n}_N^{\mathrm{in}}] = \int_{\Rd} \underline{n}_N^{\mathrm{in}}\log{\underline{n}_N^{\mathrm{in}}} - \underline{n}_N^{\mathrm{in}} \dx x.
    \end{equation*}
    To show convergence towards $\mathcal{H}[\underline{n}^{\mathrm{in}}]$, we investigate the logarithmic term. Since $\underline{n}_N^{\mathrm{in}}$ is bounded in $L^\infty(\Rd)$ and converges to $\underline{n}^{\mathrm{in}}$ in $L^1(\Rd)$, it follows from the dominated convergence theorem that, for a subsequence, $\underline{n}_N^{\mathrm{in}}\log{\underline{n}_N^{\mathrm{in}}}$ converges strongly to $\underline{n}^{\mathrm{in}}\log{\underline{n}_N^{\mathrm{in}}}$ in $L^2$. 
    Indeed, there is a subsequence such that $\underline{n}_N^{\mathrm{in}}$ converges almost everywhere on $\R^d$ and a function $h_1\in L^1(\Rd)$ such that $\underline{n}_N^{\mathrm{in}} \leq h_1$ almost everywhere.
    Since $\underline{n}_N^{\mathrm{in}}$ converges also in $L^2(\Rd)$, we can choose a further subsequence and a function $h_2 \in L^2(\Rd)$ with $\underline{n}_N^{\mathrm{in}}\leq h_2$ almost everywhere.  Note that
    \begin{equation*}
        \abs*{\underline{n}_N^{\mathrm{in}}\log{\underline{n}_N^{\mathrm{in}}}} \leq \sqrt{\underline{n}_N^{\mathrm{in}}} + (\underline{n}_N^{\mathrm{in}})^2 \leq \sqrt{h_1} + \norm*{\underline{n}_N^{\mathrm{in}}}_{L^\infty(\Rd)} h_2.
    \end{equation*}
    Since $\sqrt{h_1}+h_2 \in L^2(\Rd)$, we obtain the desired convergence in $L^2$, by dominated convergence.
  
     Now, using the second-moment control, we deduce convergence in $L^1$ as follows. Let $1>\delta>0$ be given and let $K_\delta = B_{R(\delta)}(0) \subset{\R^d}$ be a ball such that
        \begin{equation}
        \label{eq:EntropyMomentBound}
            \sup_{N\geq 2}\,\int_{\Rd\setminus K_\delta} \underline{n}_N^{\mathrm{in}}\abs*{\log{\underline{n}_N^{\mathrm{in}}}}\dx x < \delta/3.
        \end{equation}
Indeed, this is possible which can be established as follows. First, let us observe that
\begin{align*}
    \int_{\Rd \setminus K_\delta} \underline{n}_N^{\mathrm{in}} |\log \underline{n}_N^{\mathrm{in}}| \dx x &\leq d \int_{\Rd \setminus K_\delta}  \prt*{\underline{n}_N^{\mathrm{in}}}^{\frac{d+1}{d+3}} \dx x + \int_{\Rd \setminus K_\delta} \prt*{\underline{n}_N^{\mathrm{in}}}^2 \dx x\\
    &\leq d \int_{\Rd \setminus K_\delta}  \prt*{\underline{n}_N^{\mathrm{in}}}^{\frac{d+1}{d+3}} \dx x + \frac{C}{R(\delta)^2}.
\end{align*}
Note that
\begin{align*}
    \int_{\Rd \setminus K_\delta} \prt*{\underline{n}_N^{\mathrm{in}}}^{\frac{d+1}{d+3}}\dx x &= \int_{\Rd \setminus K_\delta} \prt*{\underline{n}_N^{\mathrm{in}}}^{\frac{d+1}{d+3}} \frac{|x|^{\frac{2(d+1)}{d+3}}}{|x|^{\frac{2(d+1)}{d+3}}} \dx x\\
    &\leq  \left( \int_{\Rd \setminus K_\delta} \frac{1}{|x|^{d+1}}\dx x \right)^{\frac{2}{d+3}} \left( \int_{\Rd \setminus K_\delta} {\underline{n}_N^{\mathrm{in}}}|x|^2 \dx x \right)^{\frac{d+1}{d+3}},
\end{align*}
where the second term is bounded by the uniform moment estimate. The first term can be estimated by 
$$
    \left( \int_{\Rd \setminus K_\delta} \frac{1}{|x|^{d+1}}\dx x \right)^{\frac{2}{d+3}} = \left(\int_{|x|> R(\delta)} |x|^{-d-1} \dx  x\right)^{\frac{2}{d+3}} \approx R(\delta)^{-2/(d+3)}.
$$
Choosing $R(\delta)\approx \delta^{-\frac{d+3}{2}}$, we deduce~\eqref{eq:EntropyMomentBound}.
Now, in conjunction with the $L^2$-convergence above, we find $N_0 \in N$ such that 
\begin{equation*}    \norm{\underline{n}_N^{\mathrm{in}}\log{\underline{n}_N^{\mathrm{in}}} - \underline{n}^{\mathrm{in}}\log{\underline{n}^{\mathrm{in}}}}_{L^2(\Rd)} < \frac{\delta}{3 |K_\delta|^{1/2}},
\end{equation*}
    for any $N\geq N_0$. Then,
    \begin{align*}
    \norm{\underline{n}_N^{\mathrm{in}}\log{\underline{n}_N^{\mathrm{in}}} - \underline{n}^{\mathrm{in}}\log{\underline{n}^{\mathrm{in}}}}_{L^1(\Rd)} &\leq \norm{\underline{n}_N^{\mathrm{in}}\log{\underline{n}_N^{\mathrm{in}}} - \underline{n}^{\mathrm{in}}\log{\underline{n}^{\mathrm{in}}}}_{L^1(K_\delta)} + 2\delta/3\\
    &\leq |K_\delta|^{1/2}\norm{\underline{n}_N^{\mathrm{in}}\log{\underline{n}_N^{\mathrm{in}}} - \underline{n}^{\mathrm{in}}\log{\underline{n}^{\mathrm{in}}}}_{L^2(\Rd)} + 2\delta/3\\
    &<\delta,
    \end{align*}
    which proves the claim.

	\bigskip \underline{Ad (4).} 
	Let us now consider the growth rates. Let $n\in I$ for some compact set $I\subset\R$. For $a\in[0,1]$, let $a_N$ be a sequence constructed as before. Then
	\begin{align*}
		\max_{n\in I}\abs*{G_N(n;a) - G(n;a)} = \max_{n\in I}\abs*{G(n;a_N) - G(n;a)} \leq \norm{\partial_a G}_{L^\infty(I \times [0,1])} |a_N-a| \to 0.
	\end{align*}
 
	\bigskip \underline{Ad (5).} 
	Finally, let  $(m_N)_N$ be a sequence strongly converging to $m$ in $L^p$. Then
    \begin{align*}
	   \norm{G_N(m_N;a) - G(m;a)}_{L^p} \leq \norm{G_N(m_N;a) - G_N(m;a)}_{L^p} + \norm{G_N(m;a) - G(m;a)}_{L^p}.
    \end{align*}
    The last term converges to zero by item (4) of this proof. For the remaining term we write
    \begin{align*}
        \norm{G_N(m_N;a) - G_N(m;a)}_{L^p} \leq \alpha \norm{m_n-m}_{L^p}\to 0,
    \end{align*}
    which gives the desired convergence in $L^p$. Interpolating with the $L^\infty$ bound (since $0 \leq m_N\leq\bar n$), we conclude the proof.
\end{proof}

\subsection{Compactness of the rescaled total population density}
For the remainder of this section let $\nu = \nu_N$ be any sequence such that $\nu_N\to 0$ as $N \to \infty$. Let us stress that we impose no conditions on the speed or monotonicity of the convergence $\nu_N\to0$.
The uniform bounds in $L^\infty(0,T;L^1\cap L^\infty(\Rd))$ obtained in the proof of Lemma~\ref{lem:Existence} allow us to extract subsequences of $\underline{n}\nnun$ and $n\nnun$ which converge weakly in $L^\infty(0,T,L^p(\Rd))$, $1\leq p \leq \infty$, namely
\begin{align} 
	\label{eq:inter_mean_inviscid}
	\underline{n}\nnun \rightharpoonup\underline{n}\oinfty ,\quad \text{ and } \quad n\nnun(\cdot,\cdot\,;a) \rightharpoonup n\oinfty(\cdot,\cdot\,;a),
\end{align}
for every $a\in[0,1]$. Clearly, we have
\begin{equation*}
    \underline{n}\oinfty = \int_0^1 n\oinfty\dx a.
\end{equation*}

\begin{proposition}
    \label{lem:Wton_invicid}
	The following convergence holds:
	\begin{equation*}
		W\nnun \to \underline{n}\oinfty,
	\end{equation*}
	in $ L^{2}(0,T,L^{2}(\Rd)) $. In particular, $ \underline{n}\oinfty  \in L^{2}(0,T;H^{1}(\Rd)) $.
\end{proposition}
\begin{proof}
	The local compactness of $ W\nnun $ is established using the
	Aubin-Lions lemma. We begin by proving space regularity, using the Brinkman equation
	\begin{equation*}
		\begin{aligned}
			\int_{0}^{T} \!\!\!\int_{\Rd} & \abs{\grad W\nnun}^{2}\dx{x}\dx{t}
			= -\int_{0}^{T} \!\!\!\int_{\Rd}  W\nnun \Lap W\nnun \dx{x}\dx{t},                                           \\
			 &= - \int_{0}^{T} \!\!\!\int_{\Rd}^{} \underline{n}\nnun \Lap W\nnun
			\dx{x} \dx{t} + \int_{0}^{T} \!\!\! \int_{\Rd}^{} \left(\underline{n}\nnun-W\nnun\right)\Lap W\nnun \dx{x} \dx{t}, \\
			 & = - \int_{0}^{T} \!\!\!\int_{\Rd}^{} \underline{n}\nnun \Lap W\nnun
			\dx{x}\dx{t} - \nu \int_{0}^{T} \!\!\! \int_{\Rd}^{}\abs{\Lap W\nnun}^{2}\dx{x}\dx{t},                           \\
			 & \leq- \int_{0}^{T} \!\!\!\int_{\Rd}^{} \underline{n}\nnun \Lap W\nnun
			\dx{x}\dx{t}.
		\end{aligned}
	\end{equation*}
	Now, rearranging the entropy inequality, Lemma  \ref{lem:entropy_inequality}, we obtain
	\begin{equation*}
		\begin{aligned}
			- & \int_0^T \!\!\! \int_\Rd  \underline{n}\nnun
			  \Lap W\nnun \dx x \dx{t}\\
			 & \leq \mathcal{H}[\underline{n}_{N}^{\mathrm{in}}] -\mathcal{H}[\underline{n}\nnun](T) + \int_0^T \!\!\! \int_\Rd \log \underline{n}\nnun
			\int_{0}^{1} n\nnun G_{N}(\underline{n}\nnun;a)\dx{a} \dx x \dx t\leq C,
		\end{aligned}
	\end{equation*}
 which follows from the entropy bound in Proposition~\ref{lem:entropy-bounds}. Therefore, $\grad W\nnun $ is uniformly bounded in $ L^2(0,T;L^2(\Rd)) $.
	Next, we prove time regularity. For any $\varphi \in L^2(0,T;H^1(\Rd))$, we have
 \begin{align*}
     \abs*{\langle \partial_t W\nnun, \varphi \rangle} &= \abs*{\langle \partial_t \underline{n}\nnun, K_{\nu_N} \star \varphi \rangle} \\
     &\leq \norm{\partial_t \underline{n}\nnun}_{L^2(0,T; H^{-1}(\R))} \norm{K_{\nu_N} \star \varphi}_{L^2(0,T;H^1(\Rd))}\\
     &\leq C\norm{\varphi}_{L^2(0,T;H^1(\Rd))},
 \end{align*}
 so that $\partial_t W\nnun$ is uniformly bounded in $L^2(0,T;H^{-1}(\Rd))$.
	This gives us the local compactness of $ W\nnun $ in $ L^{2}(0,T,L^{2}(\Rd))
	$. Now we show that the second moment of $ W\nnun $ is bounded,
	\begin{equation*}
		\begin{aligned}
			\int_{0}^{T} \!\!\! \int_{\Rd} W\nnun \abs{x}^{2}  \dx{x} \dx{t}
			 & =  \int_{0}^{T} \!\!\! \int_{\Rd} \underline{n}\nnun \abs{x}^{2}  \dx{x} \dx{t} + \nu \int_{0}^{T} \!\!\! \int_{\Rd} \Lap W\nnun \abs{x}^{2}  \dx{x} \dx{t}, \\
			 & = \int_{0}^{T} \!\!\! \int_{\Rd} \underline{n}\nnun \abs{x}^{2}  \dx{x} \dx{t} + 2 d \nu \int_{0}^{T} \!\!\! \int_{\Rd}  W\nnun \dx{x} \dx{t},
		\end{aligned}
	\end{equation*}
	which is bounded by Proposition~\ref{lem:second_moment}. We thus deduce global compactness of $W\nnun$. Identifying the limit of $W\nnun$ as $\underline{n}\oinfty$ follows easily from the weak convergence $\underline{n}\nnun\rightharpoonup\underline{n}\oinfty$ and the representation $W\nnun = K_{\nu_N}\star\underline{n}\nnun$.
\end{proof}
\begin{proposition}
\label{prop:NuToZeroDensity}
	The rescaled total population density $ \underline{n}\nnun $ converges strongly to $ \underline{n}\oinfty $ in $ L^{2}(0,T;L^{2}(\Rd)).$
\end{proposition}
\begin{proof}
	Let us write
	\begin{equation*}
		\begin{aligned}
			\norm{\underline{n}\nnun - \underline{n}\oinfty}_{L^{2}(0,T;L^{2}(\Rd))} \leq
			\norm{\underline{n}\nnun - W\nnun}_{L^{2}(0,T;L^{2}(\Rd))} + \norm{\underline{n}\oinfty -W\nnun}_{L^{2}(0,T;L^{2}(\Rd))}.
		\end{aligned}
	\end{equation*}
	It only remains to estimate the first term:
	\begin{equation*}
		\begin{aligned}
			\norm{\underline{n}\nnun-W\nnun}_{L^{2}(0,T;L^{2}(\Rd))}^{2}
			 & = - \nu \int_{0}^{T} \!\!\! \int_{\Rd}^{} \left( \underline{n}\nnun -W\nnun \right) \Lap W\nnun  \dx{x} \dx{t}                                                          \\
			 & = - \nu \int_{0}^{T} \!\!\! \int_{\Rd}^{}  \underline{n}\nnun \Lap W\nnun  \dx{x} \dx{t} -  \nu \int_{0}^{T} \!\!\! \int_{\Rd}^{}  \abs{\grad W\nnun}^{2} \dx{x} \dx{t} \\
			 & \leq- \nu \int_{0}^{T} \!\!\! \int_{\Rd}^{}  \underline{n}\nnun \Lap W\nnun  \dx{x} \dx{t}.
		\end{aligned}
	\end{equation*}
	Applying the entropy inequality, Lemma \ref{lem:entropy_inequality}, as before we get
	\begin{equation*}
		\begin{aligned}
			- \nu \int_{0}^{T} \!\!\! \int_{\Rd}  \underline{n}\nnun \Lap W\nnun  \dx{x} \dx{t} \leq C\nu_N.
		\end{aligned}
	\end{equation*}
	Hence,
	\begin{equation*}
		\begin{aligned}
			\norm{\underline{n}\nnun - \underline{n}\oinfty}_{L^{2}(0,T;L^{2}(\Rd))} \leq
			C\sqrt{\nu_N} + \norm{\underline{n}\oinfty -W\nnun}_{L^{2}(0,T;L^{2}(\Rd))},
		\end{aligned}
	\end{equation*}
	which goes to zero as $N \to \infty $ using Lemma~\ref{lem:Wton_invicid}.
\end{proof}

\subsection{The limit equation for the rescaled total population density}

Having obtained the strong compactness of $ \underline{n}\nnun $ and the weak compactness of $\nabla W\nnun $, we can now pass to the limit in Eq.~\eqref{eq:Brinkman_average} to obtain, in the weak sense,
\begin{equation}
\label{eq:Darcy_average}
	\frac{\partial\underline{n}\oinfty}{\partial t}  - \div(\underline{n}\oinfty \grad \underline{n}\oinfty) = \ds \int_{0}^{1}  n\oinfty G(\underline{n}\oinfty;a)\dx{a}.
\end{equation}
Testing the limiting equation by $\log(\underline{n}_{0,\infty}+ \delta) \phi$ and following an argument similar to the derivation of the entropy inequality, we obtain
\begin{equation}\label{eq:EntropyDarcyPhenotype}
	\begin{aligned}
		\mathcal{H}[\underline{n}\oinfty](T) - \mathcal{H}[\underline{n}^{\mathrm{in}}]
		 & + \int_0^T \!\!\! \int_\Rd  \abs{\grad \underline{n}\oinfty}^{2}   \dx x \dx t \\
		 & = \int_0^T \!\!\! \int_\Rd \log \underline{n}\oinfty
		\int_{0}^{1} n\oinfty G_{N}(\underline{n}\oinfty;a)\dx{a} \dx x \dx t.
	\end{aligned}
\end{equation}

\subsection{Strong convergence of the velocity}
In this subsection we will show that the $L^2$-norm of the velocity $\nabla W\nnun$ converges to the $L^2$-norm of the limit velocity $\nabla \underline{n}\oinfty$. Combining this fact with the weak convergence of the velocities from Lemma~\ref{lem:Wton_invicid}, we will deduce strong convergence.

We can compare the entropy equality, Eq. \eqref{eq:EntropyDarcyPhenotype}, and the entropy inequality,. Eq. \eqref{eq:entropy_inequality}, to deduce
\begin{equation} \label{eq:Entropy_E_I_comp}
	\begin{aligned}
		-\int_{0}^{T} \!\!\! \int_{\Rd} & \underline{n}\nnun  \Lap W\nnun   \dx{x} \dx{t}\\
        &\leq \int_{0}^{T} \!\!\! \int_{\Rd}^{}  \abs*{\grad  \underline{n}\oinfty}^{2}\dx{x}\dx{t}
		  +\mathcal{H}[\underline{n}\oinfty](T) -\mathcal{H}[\underline{n}\nnun](T)             \\
		 &\quad +\mathcal{H}[n_N^{\mathrm{in}}] -\mathcal{H}[n^{\mathrm{in}}] +\mathrm{React}(n\nnun,\underline{n}\nnun) -\mathrm{React}(n\oinfty,\underline{n}\oinfty),
	\end{aligned}
\end{equation}
where
\begin{equation*}
	\mathrm{React}(n\nnun,\underline{n}\nnun) \coloneqq  \int_0^T \!\!\! \int_\Rd \log \underline{n}\nnun
	\int_{0}^{1} n\nnun G_{N}(\underline{n}\nnun;a)\dx{a} \dx x \dx t.
\end{equation*}

We now wish to pass to the limit $N\to\infty$ in~\eqref{eq:Entropy_E_I_comp}. To this end, we investigate the terms on the right-hand side in pairs.
From Proposition~\ref{prop:Data_Strong_N} we have
\begin{equation}
\label{eq:InitialEntropyConvergence2}
    \lim_{N\to\infty} \mathcal{H}[n_N^{\mathrm{in}}] = \mathcal{H}[n^{\mathrm{in}}].
\end{equation}

To compare the final-time entropies, we first state the following proposition.

\begin{proposition}\label{prop:FinalTimeL^1}
    Up to extracting a subsequence it holds that $\underline n\nnun(T) \rightharpoonup \underline{n}\oinfty(T)$ weakly in $L^1(\Rd)$.
\end{proposition}
\begin{proof}
In light of Remark \ref{rmk:ContinuityL^1} and the uniform $L^2$-control of $\|\underline{n}_{\nu_N, N}(T)\|_{L^2(\Rd)}$, we may apply the Dunford-Pettis theorem to infer the existence of some $\chi \in L^1(\Rd)$ and a subsequence such that $\underline{n}_{\nu_N, N} \rightharpoonup \chi$, as $N\to \infty$. This limit can be identified using the equation for $\underline{n}\nnun$. 
Let $\delta>0$ and let $\xi_\delta$ be a sequence of smooth nonnegative functions approximating the continuous nonnegative piecewise linear function
\begin{equation*}
	\xi_{\delta}(s):= \begin{cases}
	1 & \text{for\;} s\in[0,T-2\delta),\\
	\frac{(T-\delta)-s}{\delta} & \text{for\;} s\in  [T-2\delta,T-\delta),\\
	0 & \text{for\;} s\in (T-\delta,T].
	\end{cases}
\end{equation*}
Then, taking $\phi \in C_c^\infty(\Rd)$ and using $\xi_\delta(t)\phi(x)$ as test function in the weak formulation of System~\eqref{eq:Brinkman_average}, we have
	\begin{align*}
		\int_0^T\!\!\!  \int_\Rd & \underline n\nnun\,\xi_\delta'\,\phi \dx x \dx t
		- \int_0^T \xi_\delta(t) \int_\Rd \underline n\nnun \nabla W\nnun \cdot \nabla \phi \dx x \dx t
		\\
		& = -\int_0^T \xi_\delta(t) \int_\Rd\int_0^1 \phi n\nnun G_{N}(\underline{n}\nun; a) \dx a \dx x \dx t -\int_\Rd \phi(x)\underline n^{\mathrm{in}}_{N}(x) \dx x.
	\end{align*}
Passing to the limit $\delta\to 0$, we obtain
	\begin{align*}
		\int_\Rd & \underline n\nnun(T)\phi \dx x
		- \int_0^T \int_\Rd \underline n\nnun \nabla W\nnun \cdot \nabla \phi \dx x \dx t
		\\
		& = -\int_0^T \int_\Rd\int_0^1 \phi n\nnun G_{N}(\underline{n}\nun; a) \dx a \dx x \dx t -\int_\Rd \phi(x)\underline n^{\mathrm{in}}_{N}(x) \dx x,
	\end{align*}
and then passing to the limit $N\to\infty$, we have
	\begin{align*}
		\int_\Rd & \chi\phi \dx x
		- \int_0^T \int_\Rd \underline n\oinfty \nabla \underline{n}\oinfty \cdot \nabla \phi \dx x \dx t
		\\
		& = -\int_0^T \int_\Rd\int_0^1 \phi n\oinfty G(\underline{n}\oinfty; a) \dx a \dx x \dx t -\int_\Rd \phi(x)\underline n^{\mathrm{in}}(x) \dx x,
	\end{align*}
where we have used Proposition~\ref{prop:Data_Strong_N}, Proposition~\ref{prop:NuToZeroDensity} and the weak convergences $\nabla W\nnun\rightharpoonup W\oinfty$, and $\underline{n}_{\nu_N, N}(T) \rightharpoonup \chi$.
However, testing the weak form of the Darcy system~\eqref{eq:Darcy_average} with the same test function $\xi_\delta(t)\phi(x)$ and passing to the limit $\delta\to 0$, we obtain
	\begin{align*}
		\int_\Rd & \underline{n}\oinfty(T)\phi \dx x
		- \int_0^T \int_\Rd \underline n\oinfty \nabla \underline{n}\oinfty \cdot \nabla \phi \dx x \dx t
		\\
		& = -\int_0^T \int_\Rd\int_0^1 \phi n\oinfty G(\underline{n}\oinfty; a) \dx a \dx x \dx t -\int_\Rd \phi(x)\underline n^{\mathrm{in}}(x) \dx x.
	\end{align*}
Juxtaposing the last two equalities, we easily deduce that $\chi = \underline{n}\oinfty(T)$ a.e.\ in $\Rd$.
\end{proof}

By convexity of the entropy functional $ \mathcal{H}[\cdot]$ and Proposition~\ref{prop:FinalTimeL^1}, we have
\begin{equation}               
\label{eq:FinalEntropyConvergence}
		\limsup_{\nu \to 0}^{} \left( \mathcal{H}[\underline{n}\oinfty](T) - \mathcal{H}[\underline{n}\nnun](T) \right) \leq 0,
\end{equation}

It remains to treat the reaction terms.
By the uniform $L^1$-control on $\underline{n}\nnun\log\underline{n}\nnun$ from Proposition~\ref{lem:entropy-bounds}, for every $ \delta > 0 $  there exists a compact set $ K=K_\delta \subset \Rd $, such that
\begin{equation*}
    \sup_{N \geq 2}\, \int_0^T \!\!\! \int_{K^{c}}\underline{n}\nnun \abs*{\log \underline{n}\nnun} \dx x \dx t \leq \delta.
\end{equation*}
Then, since the growth rate is uniformly bounded,
\begin{equation*}
    \abs*{\int_0^T \!\!\! \int_{K^c} \log \underline{n}\nnun
	\int_{0}^{1} n\nnun G_{N}(\underline{n}\nnun;a)\dx{a} \dx x \dx t} \leq C\delta,
\end{equation*}
for each $N \geq 2$ as well as for the limit quantity at $N=\infty$.

For the integral on $ K $, we apply Egorov's Theorem to find $ A\subset K $ with $ \mu(A) \leq \delta $ such that $ \underline{n}\nnun \to \underline{n}\oinfty $ uniformly in $ K\setminus A $. 
Let $ 0 < \beta < 1/e  $. By uniform convergence, on $K\setminus A$ we have (for $N$ large enough)
\begin{equation*}
	\begin{aligned}
		\underline{n}\oinfty \geq \frac{\beta}{2} \implies \underline{n}\nnun \geq \frac{\beta}{4}\quad\text{and}\quad
		\underline{n}\oinfty < \frac{\beta}{2} \implies \underline{n}\nnun < \beta.
	\end{aligned}
\end{equation*}
We now partition the compact domain $ K $ as follows,
\begin{equation*}
	K = (K \setminus A \cap \set{\underline{n}\oinfty\geq\beta/2}) \cup (K \setminus A \cap \set{\underline{n}\oinfty<\beta/2}) \cup A \equiv K_{1} \cup K_{2} \cup A.
\end{equation*}
On $K_{1}$ the sequence $\log{\underline{n}\nnun}$ converges strongly in $L^p$. Thus, together with the weak convergence of $n\nnun$ and the strong convergence of $G_N(\underline{n}\nnun,a)$ (cf.~Proposition~\ref{prop:Data_Strong_N}), we readily obtain
\begin{equation*}
	\begin{aligned}
	\int_0^T \!\!\! \int_{K_{1}}\log \underline{n}\nnun &\int_{0}^{1} n\nnun G_{N}(\underline{n}\nnun;a)\dx{a} \dx x \dx t\\
    &\to \int_0^T \!\!\! \int_{K_{1}}\log \underline{n}\oinfty \int_{0}^{1} n\oinfty G(\underline{n}\oinfty;a)\dx{a} \dx x \dx t.
	\end{aligned}
\end{equation*}
On the set $ K_{2} $, using the uniform bound for $ G_{N} $ and $
	\underline{n}\nnun < \beta $, we have
\begin{equation*}
		\bigg|\int_0^T \!\!\! \int_{K_{2}}\log \underline{n}\nnun
		\int_{0}^{1} n\nnun G_{N}(\underline{n}\nnun;a)\dx{a} \dx x \dx t\bigg|
		\leq C\max_{[0,1]\times [0,\bar{n}]}\abs*{G(n;a)} \mu(K)\beta \abs*{\log\beta}.
\end{equation*}
Lastly, on the set $ A $
\begin{equation*}
	\begin{aligned}
		\bigg| \int_0^T \!\!\! \int_{A}\log \underline{n}\nnun \int_{0}^{1} n\nnun G_{N}(\underline{n}\nnun;a)\dx{a} \dx x \dx t\bigg| \leq C \delta.
	\end{aligned}
\end{equation*}
Let us stress the the last two bounds hold for $N\geq 2$ and for $N=\infty$.
Combining all the above results, we conclude that
\begin{equation*}
	\begin{aligned}
		&\lim_{N\to\infty}\abs*{\mathrm{React}(n\nnun,\underline{n}\nnun) - \mathrm{React}(n\oinfty,\underline{n}\oinfty)}
		  \leq C(\delta + \mu(K_\delta)\beta|\log\beta|).
	\end{aligned}
\end{equation*}
Sending $\beta\to 0$ and then $\delta\to0$, we deduce that
\begin{equation}
\label{eq:ReactionTermsConvergence}
    \lim_{N\to\infty}\mathrm{React}(n\nnun,\underline{n}\nnun) = \mathrm{React}(n\oinfty,\underline{n}\oinfty).
\end{equation}

Finally, let us recall from the proof of Proposition~\ref{lem:Wton_invicid} that
\begin{equation}
\label{eq:Auxiliary}
		-\int_{0}^{T} \!\!\! \int_{\Rd} \underline{n}\nnun  \Lap W\nnun \dx{x} \dx{t}
		\geq\int_{0}^{T} \!\!\! \int_{\Rd} \abs*{\grad  W\nnun}^{2}  \dx{x} \dx{t}.
\end{equation}

Therefore, using~\eqref{eq:Auxiliary}, and results~\eqref{eq:InitialEntropyConvergence2},~\eqref{eq:FinalEntropyConvergence}, and~\eqref{eq:ReactionTermsConvergence}, we can pass to the limit $N\to\infty$ in Eq.~\eqref{eq:Entropy_E_I_comp} to obtain
\begin{equation*}
	\begin{aligned}
		\limsup_{\nu \to 0}^{}    \int_{0}^{T} \!\!\! \int_{\Rd} \abs*{\grad  W\nnun}^{2}  \dx{x} \dx{t} \leq \int_{0}^{T} \!\!\! \int_{\Rd}^{}  \abs*{\grad  \underline{n}\oinfty}^{2}\dx{x}\dx{t},
	\end{aligned}
\end{equation*}
combining which with the weak convergence $\nabla W\nnun\rightharpoonup \nabla n\oinfty$, yields strong convergence of the velocity $\nabla W\nnun$.

\subsection{Conclusion of the proof}
Using the weak convergence~\eqref{eq:inter_mean_inviscid} of the interpolated density $n\nnun$, the strong convergence of the rescaled population density $\underline{n}\nnun$, the strong convergence of the velocity $\nabla W\nnun\to\nabla\underline{n}\oinfty$, and the results of Proposition~\ref{prop:Data_Strong_N} (for the convergence of the initial data and the reaction term), it is straightforward to pass to the limit in the weak formulation of System~\eqref{eq:Brinkman_a} to obtain that $n\oinfty$ satisfies System~\eqref{eq:Darcy_Phenotype}, in the weak sense, with velocity $\nabla\underline{n}\oinfty$. Thus, the proof of Theorem~\ref{thm:JointLimit} is complete. \qed

\section{The limit $N\to \infty$}
\label{sec:continuous-phenotype}
This section is devoted to proving Theorem~\ref{thm:PhenotypeLimit}. To this end, let $\nu>0$ be fixed and let $(n\i\nun, W\nun)_{i=1}^N$ be the weak solution constructed in Lemma~\ref{lem:Existence} as the limit of the approximate sequence $(n\i\epsnun, W\epsnun)_{i=1}^N$.
With the uniform bounds guaranteed by Lemma~\ref{lem:Existence}, we can extract a subsequence such that
\begin{equation*}
	n\nun(\cdot,\cdot\,;a) \stackrel{\star}{\rightharpoonup} n_{\nu,\infty}(\cdot,\cdot\,;a)
\end{equation*}
weakly$-*$ in $L^\infty(0,T;L^p(\Rd))$ for any $p\in[1,\infty]$.
Defining $\underline{n}_{\nu,\infty}:=\int_0^1 n_{\nu,\infty} \dx a$, we immediately obtain that $\underline{n}\nun$ converges weakly to $\underline{n}_{\nu,\infty}$ in the same range of topologies.

Now define $W_{\nu,\infty}$ as the solution to the Brinkman equation
\begin{equation*}
	-\nu\Delta W_{\nu,\infty} + W_{\nu,\infty} = \underline{n}_{\nu,\infty}.
\end{equation*}
Clearly $W\nun$ converges weakly to $W_{\nu,\infty}$.

To pass to the limit $N\to\infty$ in the weak formulation of System~\eqref{eq:Brinkman_a}, we will demonstrate that the sequences $(\underline{n}\nun)_N$ and $(\nabla W\nun)_N$ converge strongly.
\begin{proposition}
	\label{prop:W_Strong_N}
	The sequence $(W\nun)_N$ converges, as $N\to\infty$, to $W_{\nu,\infty}$ strongly in $L^1(0,T;L^1(\R^d))$.
\end{proposition}
\begin{proof}
	This convergence is obtained similarly as in Proposition~\ref{prop:W_eps_Strong}. The gradient $\nabla W\nun$ is uniformly bounded in $L^\infty(0,T;L^p(\Rd))$ for any $1\leq p \leq \infty$, the time derivative $\partial_t W\nun = \partial_t \underline{n}\nun \star K_\nu$ is uniformly bounded in $L^2(0,T;H^{-1}(\Rd))$ using the uniform bound on $\partial_t \underline{n}\nun$ from Corollary~\ref{cor:EpsConvergence}, and the second-moment control in Proposition~\ref{lem:second_moment} provides equi-tightness.
\end{proof}

\begin{proposition}
	The sequence $(\nabla W\nun)_N$ converges to $\nabla W_{\nu,\infty}$ strongly in $L^2(0,T;L^2_{\mathrm{loc}}(\Rd))$.
\end{proposition}
\begin{proof}
	This follows from the Aubin-Lions lemma. From the Brinkman equation, we have $D^2 W\nun$ bounded in $L^2(0,T;L^2(\Rd))$. Since $\nu$ is fixed, the bound for the time derivative is obtained from $\partial_t\prt{\partial_{x_k} W\nun} = \partial_{x_k}K_\nu \star \partial_t \underline{n}\nun$ and Corollary~\ref{cor:EpsConvergence} again.
\end{proof}

\subsection{Compactness of the rescaled total population density}
We now revisit the proof of Proposition~\ref{prop:n_epsnu_compact} to deduce strong convergence of the rescaled total population density.
\begin{proposition}
	\label{prop:Strong_N}
	The sequence $(\underline{n}\nun)_N$ converges strongly in $L^1(\Rd\times(0,T))$.
\end{proposition}
\begin{proof}
    Recall from the proof of Proposition~\ref{prop:n_epsnu_compact} (cf.\ Eq.~\eqref{eq:Compactness_Gronwall}) that
    \begin{align*}
		&\int_{\R^{2d}} \calK_h(x-y)\abs*{\underline{n}\epsnun(x) - \underline{n}\epsnun(y)}\dx x\dx y\\
        &\leq C\frac{\epsilon}{h^2} + C\abs*{\log h}^{1/2} + C\int_0^T\int_{\R^{2d}}\calK_h(x-y)\abs*{W\epsnun(x) - W\epsnun(y)} \dx x\dx y\dx t\\
        &\quad + C\int_{\R^{2d}} \calK_h(x-y)\int_0^1\abs*{n_{\epsilon,N}^{\mathrm{in}}(x) - n_{\epsilon,N}^{\mathrm{in}}(y)}\dx a\dx x\dx y     \\
		& \quad + C\int_{\R^{2d}} \calK_h(x-y)\abs*{\underline{n}_{\epsilon,N}^{\mathrm{in}}(x) - \underline{n}_{\epsilon,N}^{\mathrm{in}}(y)}\dx x\dx y,
	\end{align*}
    where all the constants are independent of $\epsilon$ and $N$.
    However, we already know that $\underline{n}\epsnun$ converges to $\underline{n}\nun$ in $L^1(\Rd \times (0,T))$, as $\epsilon\to0$. Therefore, passing to the limit in $\epsilon$, we obtain
    \begin{align*}
		&\int_{\R^{2d}} \calK_h(x-y)\abs*{\underline{n}\nun(x) - \underline{n}\nun(y)}\dx x\dx y\\
        &\leq C\abs*{\log h}^{1/2} + C\int_0^T\int_{\R^{2d}}\calK_h(x-y)\abs*{W\nun(x) - W\nun(y)} \dx x\dx y\dx t\\
        &\quad + C\int_{\R^{2d}} \calK_h(x-y)\int_0^1\abs*{n_{N}^{\mathrm{in}}(x) - n_{N}^{\mathrm{in}}(y)}\dx a\dx x\dx y     \\
		& \quad + C\int_{\R^{2d}} \calK_h(x-y)\abs*{\underline{n}_{N}^{\mathrm{in}}(x) - \underline{n}_{N}^{\mathrm{in}}(y)}\dx x\dx y.
	\end{align*}
    Subsequently, using Proposition~\ref{prop:W_Strong_N} and Proposition~\ref{prop:Data_Strong_N}, we infer
	\begin{equation*}
		\lim_{h\to0}\,\limsup_{N\to\infty}\,|\log h|^{-1}\,\int_0^T\int_{\R^{2d}} \calK_h(x-y)\abs*{\underline{n}\nun(x) - \underline{n}\nun(y)}\dx x\dx y \dx t = 0.
	\end{equation*}
    Applying Lemma~\ref{lem:CompactnessCriterion}, we deduce the result. 
\end{proof}

\subsection{Passing to the limit}
We are now ready to perform the limit passage in the weak formulation of Eq.~\eqref{eq:Brinkman_average}, see Definition~\ref{def:wk_sol_brinkman}. Note that to pass to the limit in the reaction term we use Proposition~\ref{prop:Data_Strong_N}.
We can also pass to the limit in the Brinkman equation to obtain
\begin{equation*}
	-\nu\Delta W_{\nu,\infty} + W_{\nu,\infty} = \underline{n}_{\nu,\infty},
\end{equation*}
almost everywhere in $\Rd\times(0,T)$.
We have thus obtained a weak solution $(n_{\nu,\infty}, W_{\nu,\infty})$ of System~\eqref{eq:Brinkman_Phenotype} as the limit of weak solutions $(n\nun, W\nun)$ of System~\eqref{eq:Brinkman_a}, which concludes the proof of Theorem~\ref{thm:PhenotypeLimit}.

We conclude this section by observing that the entropy inequality holds also for System~\eqref{eq:Brinkman_Phenotype}.

\begin{proposition}
	The weak solution $(n_{\nu,\infty}, W_{\nu,\infty})$ of System~\eqref{eq:Brinkman_Phenotype} obtained above satisfies the following entropy inequality
	\begin{equation}
		\label{eq:entropy_inequality_phenotype}
		\begin{split}
			\mathcal{H}[\underline{n}_{\nu,\infty}](T) - \mathcal{H}[\underline{n}^{\mathrm{in}}]
			 & - \int_0^T \!\!\! \int_\Rd  \underline{n}_{\nu,\infty} \Lap W_{\nu,\infty} \dx x \dx t \\
			 & \leq\int_0^T \!\!\! \int_\Rd \log \underline{n}_{\nu,\infty}
			\int_{0}^{1} n_{\nu,\infty} G(\underline{n}_{\nu,\infty};a)\dx{a} \dx x \dx t.
		\end{split}
	\end{equation}
\end{proposition}
\begin{proof}
	This follows from taking the limit $N\to\infty$ in the inequality for the localised entropies, Eq.~\eqref{eq:Localised_Entropy_Inequality}. For the entropy terms we use the strong convergence property derived in Proposition~\ref{prop:TimeContinuity} and Proposition~\ref{prop:Data_Strong_N}, respectively. The critical term is the reaction term, since we only have weak convergence of $n\nun$. However, using the same approach as in Section~\ref{sec:JointLimit} (Egorov's theorem and partition of the domain $\supp\phi$), we can show that
    \begin{align*}
        \int_0^T \!\!\! \int_\Rd \phi \log{\underline{n}\nun}
			\int_0^1 & n\nun G_{N}(\underline n\nun;a) \dx a \dx x \dx t \\
            &\to\int_0^T \!\!\! \int_\Rd \phi \log \underline{n}_{\nu,\infty}
			\int_{0}^{1} n_{\nu,\infty} G(\underline{n}_{\nu,\infty};a)\dx{a} \dx x \dx t.
    \end{align*}
    Considering a sequence of cut-off converging to unity, we remove the localisation and deduce Eq. ~\eqref{eq:entropy_inequality_phenotype}.
\end{proof}

With the above entropy inequality and the uniform bounds from Lemma~\ref{lem:Existence}, we can easily follow the strategy explained in Section~\ref{sec:JointLimit} to deduce the following result on the inviscid limit at the continuous phenotype level.

\begin{corollary}
    Let $(n_{\nu,\infty}, W_{\nu,\infty})$ be the weak solution of System~\eqref{eq:Brinkman_Phenotype} obtained above. Then, there exists a function $n_{0, \infty}: \Rd \times [0,T] \times [0,1] \to [0, \infty)$ with 
    \begin{equation*}
        n_{0, \infty}(\cdot, \cdot\,; a)\in L^\infty(0,T; L^1 \cap L^\infty(\Rd)),\quad \text{for almost every $a \in [0,1]$},
    \end{equation*}
    and
    \begin{equation*}
        \underline{n}_{0, \infty} \in L^2(0,T;H^1(\Rd)),
    \end{equation*}
    such that, up to a subsequence, 
    \begin{align*}
        n_{\nu, \infty} &\stackrel {\star}{\rightharpoonup} n_{0,\infty} \quad\text{in } L^\infty(0,T;L^p(\Rd)), 1\leq p\leq \infty,\\
        \underline{n}_{\nu, \infty} &\to \underline{n}_{0,\infty} \quad\text{in } L^2(0,T;L^2(\Rd)),\\
        W_{\nu, \infty} &\to \underline{n}_{0,\infty} \quad\text{in } L^2(0,T;H^1(\Rd)),
    \end{align*}
    The limit function $n_{0,\infty}$ is a weak solution to System~\eqref{eq:Darcy_Phenotype} in the sense of Definition \ref{def:wk_sol_darcy}.
\end{corollary}

\subsection{The phenotype limit at the Darcy level}
We conclude this section with a short discussion of the limit $N\to\infty$ for the inviscid system~\eqref{eq:Darcy_a}. 
Let $(n\oun)_N$ be a sequence of nonnegative weak solutions of System~\eqref{eq:Darcy_a}. The existence of such solutions can be obtained by taking the limit $\nu\to0$ in the family of solutions $(n\nun)_\nu$ to the Brinkman system~\eqref{eq:Brinkman_a}. In particular, these solutions  are uniformly bounded in $L^\infty(0,T;L^1\cap L^\infty(\Rd))$. Hence, up to a subsequence, we have
\begin{equation*}
    n\oun(\cdot,\cdot;a)\rightharpoonup n\oinfty(\cdot,\cdot;a)\quad \text{in } L^2(0,T;L^2(\Rd)),
\end{equation*}
for each $a\in[0,1]$, for some function $n\oinfty \in L^\infty(0,T;L^1\cap L^\infty(\Rd))$.
We then have
\begin{equation*}
    \underline{n}\oun := \int_0^1 n\oun\dx a \rightharpoonup \int_0^1 n\oinfty\dx a =: \underline{n}\oinfty,
\end{equation*}
weakly in $L^2(0,T;L^2(\Rd))$.
The rescaled total density $\underline{n}\oun$ satisfies the following equation:
\begin{equation*}
	\frac{\partial\underline{n}\oun}{\partial t}  - \div(\underline{n}\oun \grad \underline{n}\oun) = \ds \int_{0}^{1}  n\oun G(\underline{n}\oun;a)\dx{a}.
\end{equation*}
Testing this equation by $\log{\underline{n}\oun}$ and bounding the entropy terms, we obtain

\begin{equation}\label{eq:EntropyDarcyPhenotype2}
	\begin{aligned}
	\int_0^T \!\!\! \int_\Rd  \abs{\grad \underline{n}\oun}^{2}   \dx x \dx t
		 & \leq \abs{\mathcal{H}[\underline{n}\oinfty](T)} + \abs{\mathcal{H}[\underline{n}^{\mathrm{in}}]}\\
        &\quad+ \int_0^T \!\!\! \int_\Rd \abs{\log \underline{n}\oun}
		\int_{0}^{1} n\oun G_{N}(\underline{n}\oun;a)\dx{a} \dx x \dx t\\
        &\leq C,
	\end{aligned}
\end{equation}
where the constant is independent of $N$.
This implies that $\nabla \underline{n}\oun$ is uniformly bounded in $L^2(0,T;L^2(\Rd))$. In conjunction with the uniform bound $\partial_t\underline{n}\oun \in L^2(0,T;H^{-1}(\Rd))$, we deduce that there exists a subsequence such that $\underline{n}\oun$ converges to $\underline{n}\oinfty$ in the $L^2$-norm. Along the same subsequence we also have $\nabla \underline{n}\oun\rightharpoonup\nabla\underline{n}\oinfty$ in $L^2(0,T;L^2(\Rd))$.
These convergences are sufficient to pass to the limit in the weak formulation of System~\eqref{eq:Darcy_a} to deduce that $n\oinfty$ satisfies the weak formulation of System~\eqref{eq:Darcy_Phenotype}.

\section{Concluding Remarks}
\label{sec:conclusion}
In this work, we proposed tissue growth models featuring  $N$  subpopulations governed by viscoelastic interactions through Brinkman's law. Our investigation focused on two distinct limit processes. First, in the joint limit as $\nu \to 0$ and $N \to \infty$, we recovered the inviscid tissue growth model with a continuous phenotype variable, as introduced in \cite{Dav2023}. The second limit considered the number of phenotype traits approaching infinity ($N \to \infty$) for both $\nu > 0$ and $\nu = 0$. For $\nu > 0$, we derived a viscoelastic tissue growth model with a continuous phenotype variable. In the case of $\nu = 0$, we again obtained the inviscid tissue growth model with continuous traits, as studied in \cite{Dav2023}. Additionally, the inviscid limit $\nu \to 0$ for a fixed finite number of phenotypes is a straightforward generalisation of \cite{DDMS2024}. Thus, our work provides a comprehensive framework that elucidates the relationships between these four modelling paradigms, as depicted in the diagram presented in the introduction.\\

To the best of our knowledge, our results constitute the first rigorous `phenotype-to-infinity' limits in this context, opening up exciting future research avenues. These include the incorporation of convective effects and the exploration of broader classes of constitutive pressure laws, for which we believe our findings will also hold true.

\bigskip
\textsc{Acknowledgement}\\
T.D. acknowledges the support of the Polish National Agency for
Academic Exchange (NAWA), agreement no. BPN/BDE/2023/1/00011/U/00001. M.M. and M.S. would like to acknowledge the support of the German Academic Exchange Service (DAAD) Project ID 57699543.

\bibliographystyle{abbrv}
\bibliography{lit}

\end{document}